\documentclass[12pt,reqno]{article}

\usepackage[usenames,dvipsnames]{color}
\usepackage{amssymb}
\usepackage{graphicx}
\usepackage{amscd}

\usepackage[colorlinks=true,
linkcolor=webgreen,
filecolor=webbrown,
citecolor=webgreen]{hyperref}

\definecolor{webgreen}{rgb}{0,.5,0}
\definecolor{webbrown}{rgb}{.6,0,0}

\usepackage{color}
\usepackage{fullpage}
\usepackage{float}

\usepackage{graphics,amsmath,amssymb}
\usepackage{amsthm}
\usepackage{amsfonts}
\usepackage{latexsym}
\usepackage{epsf}

\newcommand{\seqnum}[1]{\href{http://oeis.org/#1}{\color{ProcessBlue}{\underline{#1}}}}

\newcommand{\citep}{\cite}
\newcommand{\cf}[0]{cf.\ } 
\renewcommand{\emph}[1]{\textit{#1}} 

\theoremstyle{plain}
\newtheorem{theorem}{Theorem}[section]
\newtheorem{lemma}[theorem]{Lemma}
\newtheorem{cor}[theorem]{Corollary}
\newtheorem{prop}[theorem]{Proposition}

\theoremstyle{definition}
\newtheorem{definition}[theorem]{Definition}
\newtheorem{example}[theorem]{Example}
\newtheorem{remark}[theorem]{Remark}

\usepackage{caption}
\usepackage{subcaption}
\newcommand{\Iverson}[1]{\ensuremath{\left[#1\right]_{\delta}}} 

\usepackage{enumitem}
\setlist[itemize]{leftmargin=0.65in}

\usepackage{rotating} 

\usepackage{diagbox}

\usepackage{ifthen}
\newcommand{\Hn}[2]{
     \ifthenelse{\equal{#2}{1}}{H_{#1}}{H_{#1}^{\left(#2\right)}}
}
\newcommand{\floor}[1]{\ensuremath{\left\lfloor #1 \right\rfloor}} 

\DeclareMathOperator{\od}{od} 
\DeclareMathOperator{\PP}{PP} 
\DeclareMathOperator{\PL}{PL} 

\allowdisplaybreaks

\begin{document}

\begin{center}
\vskip 1cm{\LARGE\bf 
       Exact Formulas for the Generalized \\ 
       \vskip 0.1in
       Sum-of-Divisors Functions} 
\vskip 1cm
\large
Maxie D. Schmidt \\
School of Mathematics \\ 
Georgia Institute of Technology \\ 
Atlanta, GA 30332 \\
USA \\ 
\href{mailto:maxieds@gmail.com}{\tt maxieds@gmail.com} \\ 
\href{mailto:mschmidt34@gatech.edu}{\tt mschmidt34@gatech.edu} \\
\end{center}

\vskip .2 in

\begin{abstract}
We prove new exact formulas for the generalized sum-of-divisors functions, 
$\sigma_{\alpha}(x) := \sum_{d|x} d^{\alpha}$. 
The formulas for $\sigma_{\alpha}(x)$ when $\alpha \in \mathbb{C}$ is fixed and 
$x \geq 1$ involves a finite sum over all of the prime factors $n \leq x$ and 
terms involving the $r$-order harmonic number sequences and the Ramanujan sums $c_d(x)$. 
The generalized harmonic number sequences correspond to the 
partial sums of the Riemann zeta function when $r > 1$ and are related to the 
generalized Bernoulli numbers when $r \leq 0$ is integer-valued. 

A key part of our new expansions of the Lambert series generating functions for the 
generalized divisor functions is formed by taking 
logarithmic derivatives of the 
cyclotomic polynomials, $\Phi_n(q)$, which completely factorize the 
Lambert series terms $(1-q^n)^{-1}$ into irreducible polynomials in $q$. 
We focus on the computational aspects of these exact expressions, including their 
interplay with experimental mathematics, and comparisons of the new formulas for 
$\sigma_{\alpha}(n)$ and the summatory functions $\sum_{n \leq x} \sigma_{\alpha}(n)$.
\end{abstract} 



\noindent 
\textit{Revised: } \today

\section{Introduction} 

\subsection{Lambert series generating functions} 

We begin our search for interesting formulas for the 
\emph{generalized sum-of-divisors} functions, 
$\sigma_{\alpha}(n)$ for $\alpha \in \mathbb{C}$, by 
expanding the partial sums of the Lambert series which generate these functions 
in the form of \citep[\S 17.10]{HARDYWRIGHTNUMT} \citep[\S 27.7]{NISTHB} 
\begin{align} 
\label{eqn_LambertSeriesOGF_def} 
\widetilde{L}_{\alpha}(q) & := \sum_{n \geq 1} \frac{n^{\alpha} q^n}{1-q^n} = 
     \sum_{m \geq 1} \sigma_{\alpha}(m) q^m,\ |q| < 1. 
\end{align} 
In particular, we arrive at new expansions of the 
partial sums of Lambert series generating functions in 
\eqref{eqn_LambertSeriesOGF_def} which generate our special arithmetic sequences as 
\begin{align} 
\label{eqn_GenSODFn_DivisorSumDef_v1} 
\sigma_{\alpha}(x) & = [q^x]\left(\sum_{n=1}^x \frac{n^{\alpha} q^n}{1-q^n} 
     \right) = 
     \sum_{d|x} d^{\alpha},\ \alpha \in \mathbb{Z}^{+}.
\end{align} 

\subsection{Factoring partial sums into irreducibles} 

The technique employed in this article using the Lambert series expansions 
in \eqref{eqn_LambertSeriesOGF_def} is to expand by repeated 
use of the properties of the well-known sequence of 
\emph{cyclotomic polynomials}, $\Phi_n(q)$, 
defined by \citep[\S 3]{DAVENPORT-MULTNUMT} \cite[\S 13.2]{CYCLOTOMIC-POLYS} 
\begin{align} 
\label{eqn_CyclotomicPoly_ProductDef} 
\Phi_n(q) & := \prod_{\substack{1 \leq k \leq n \\ \gcd(k, n) = 1}} \left( 
     q - e^{2\pi\imath \frac{k}{n}}\right). 
\end{align} 
For each integer $n \geq 1$ we have the factorizations 
\begin{align} 
\label{eqn_CyclotomicFactorization_stmt_v1} 
q^n-1 & = \prod_{d|n} \Phi_d(q), 
\end{align} 
or equivalently that 
\begin{align} 
\label{eqn_KeyCyclotomicPolyProdIdent_ForLLambdaFn} 
\Phi_n(q) & = \prod_{d|n} (q^d-1)^{\mu(n/d)}, 
\end{align} 
where $\mu(n)$ denotes the \emph{M\"obius function}. 
If $n = p^m r$ with $p$ prime and $\gcd(p, r) = 1$, we have the identity that 
$\Phi_n(q) = \Phi_{pr}(q^{p^{m-1}})$. 
In later results stated in the article, 
we use the known expansions of the cyclotomic polynomials 
which reduce the order $n$ of the polynomials by exponentiation of the 
indeterminate $q$ when $n$ contains a factor of a prime power. 
A short list summarizing these transformations is given as follows 
for $p$ and odd prime, $k \geq 1$, and where $p \nmid r$:
\begin{equation} 
\label{eqn_CyclotomicPolysPrimeProperties} 
\Phi_{2p}(q) = \Phi_p(-q), \Phi_{p^k}(q) = \Phi_p\left(q^{p^{k-1}}\right), 
\Phi_{p^k r}(q) = \Phi_{pr}\left(q^{p^{k-1}}\right), 
\Phi_{2^k}(q) = q^{2^{k-1}}+1, 
\end{equation} 
The next definitions to expand our Lambert series generating 
functions further by factoring its terms by the cyclotomic polynomials\footnote{ 
     \textit{Special notation}: 
     \emph{Iverson's convention} compactly specifies 
     boolean-valued conditions and is equivalent to the 
     \emph{Kronecker delta function}, $\delta_{i,j}$, as 
     $\Iverson{n = k} \equiv \delta_{n,k}$. 
     Similarly, $\Iverson{\mathtt{cond = True}} \equiv 
     \delta_{\mathtt{cond}, \mathtt{True}} \in \{0, 1\}$, 
     which is $1$ if and only if \texttt{cond} is true, 
     in the remainder of the article. 
}. 

\begin{definition}[Notation and logatithmic derivatives] 
\label{def_NotationProductsLogDer} 
For $n \geq 2$ and indeterminate $q$, we define the following 
rational functions related to the logarithmic derivatives of the cyclotomic 
polynomials: 
\begin{align} 
\label{eqn_Pin_Cyclotomic2_defs}
\Pi_n(q) & := \sum_{j=0}^{n-2} \frac{(n-1-j) q^j(1-q)}{(1-q^n)} = 
     \frac{(n-1) - n q - q^n}{(1-q)(1-q^n)} \\ 
\notag 
\widetilde{\Phi}_n(q) & := \frac{1}{q} \cdot \frac{d}{dw}\left[\log \Phi_n(w) \right] 
     \Bigr|_{w \rightarrow \frac{1}{q}}. 
\end{align} 
For any natural number $n \geq 2$ and prime $p$, we use 
$\nu_p(n)$ to denote the largest power of $p$ dividing $n$. 
If $p \nmid n$, then $\nu_p(n) = 0$ and if 
$n = p_1^{\gamma_1} p_2^{\gamma_2} \cdots p_k^{\gamma_k}$ is the prime 
factorization of $n$ then $\nu_{p_i}(n) = \gamma_i$. 
That is, $\nu_p(n)$ is the valuation function indicating the exact non-negative 
exponent of the prime $p$ dividing any $n \geq 2$. 
In the notation that follows, we consider sums indexed by $p$ to be summed over 
only the primes $p$ by convention unless specified otherwise. 
Finally, we define the function $\widetilde{\chi}_{\PP}(n)$ 
to denote the indicator function 
of the positive natural numbers $n$ which are not of the form 
$n = p^k, 2p^k$ for any primes $p$ and exponents $k \geq 1$. 
The conventions which make this definition accessible will become clear in the 
next subsections.
\end{definition} 

\begin{table}[ht!] 

\newcommand{\C}[2]{\widetilde{\Phi}_{#1}\left(#2\right)} 
\renewcommand{\arraystretch}{1.5} 

\begin{equation*} 
\resizebox{\linewidth}{!}{
$\begin{array}{|c|l|l|l|} \hline 
 n & \parbox{5cm}{\small\bf{Lambert Series Expansions}} 
 \left(\frac{n q^n}{1-q^n} + n - \frac{1}{1-q}\right) & 
 \parbox{3cm}{\small\bf Formula \\ Expansions} & 
 \parbox{3cm}{\small\bf Reduced-Index \\ Formula} \\ \hline 
 2 & \frac{1}{1+q} & \scriptstyle \C{2}{q} & \scriptstyle \text{- -} \\
 3 & \frac{2+q}{1+q+q^2} & \scriptstyle \C{3}{q} & \scriptstyle \text{- -} \\
 4 & \frac{1}{1+q}+\frac{2}{1+q^2} & \scriptstyle \C{2}{q} + \C{4}{q} & \scriptstyle \C{2}{q} + 2\C{2}{q^2} \\
 5 & \frac{4+3 q+2 q^2+q^3}{1+q+q^2+q^3+q^4} & \scriptstyle \C{5}{q} & \scriptstyle \text{- -} \\
 6 & \frac{1}{1+q}+\frac{2-q}{1-q+q^2}+\frac{2+q}{1+q+q^2} & 
     \scriptstyle \C{2}{q} + \C{3}{q} + \C{6}{q} & \scriptstyle \text{- -} \\
 7 & \frac{6+5 q+4 q^2+3 q^3+2 q^4+q^5}{1+q+q^2+q^3+q^4+q^5+q^6} & 
     \scriptstyle \C{7}{q} & \scriptstyle \text{- -} \\
 8 & \frac{1}{1+q}+\frac{2}{1+q^2}+\frac{4}{1+q^4} & 
     \scriptstyle \C{2}{q} + \C{4}{q} + \C{8}{q} & 
     \scriptstyle \C{2}{q} + 2\C{2}{q^2} + 4\C{2}{q^4} \\
 9 & \frac{2+q}{1+q+q^2}+\frac{3 \left(2+q^3\right)}{1+q^3+q^6} & 
     \scriptstyle \C{3}{q} + \C{9}{q} & 
     \scriptstyle \C{3}{q} + 3\C{3}{q^2} \\
 10 & \frac{1}{1+q}+\frac{4-3 q+2 q^2-q^3}{1-q+q^2-q^3+q^4}+\frac{4+3 q+2 q^2+q^3}{1+q+q^2+q^3+q^4} & 
      \scriptstyle \C{2}{q} + \C{5}{q} + \C{10}{q} & 
      \scriptstyle \text{- -} \\
 11 & \frac{10+9 q+8 q^2+7 q^3+6 q^4+5 q^5+4 q^6+3 q^7+2 q^8+q^9}{1+q+q^2+q^3+q^4+q^5+q^6+q^7+q^8+q^9+q^{10}} & 
      \scriptstyle \C{11}{q} & \scriptstyle \text{- -} \\
 12 & \frac{1}{1+q}+\frac{2}{1+q^2}+\frac{2-q}{1-q+q^2}+\frac{2+q}{1+q+q^2}-\frac{2 \left(-2+q^2\right)}{1-q^2+q^4} & 
      \scriptstyle \C{2}{q} + \C{3}{q} + \C{4}{q} & 
     \scriptstyle \C{2}{q} + 2\C{2}{q^2} + \C{3}{q} \\ 
    & & \scriptstyle \phantom{\C{2}{q}} + \C{6}{q} + \C{12}{q} & 
        \scriptstyle \phantom{\C{2}{q}} + \C{6}{q} + 2\C{6}{q} \\
 13 & \frac{12+11 q+10 q^2+9 q^3+8 q^4+7 q^5+6 q^6+5 q^7+4 q^8+3 q^9+2 q^{10}+q^{11}}{1+q+q^2+q^3+q^4+q^5+q^6+q^7+q^8+q^9+q^{10}+q^{11}+q^{12}} & 
      \scriptstyle \C{13}{q} & \scriptstyle \text{- -} \\
 14 & \frac{1}{1+q}+\frac{6-5 q+4 q^2-3 q^3+2 q^4-q^5}{1-q+q^2-q^3+q^4-q^5+q^6}+\frac{6+5 q+4 q^2+3 q^3+2 q^4+q^5}{1+q+q^2+q^3+q^4+q^5+q^6} & 
      \scriptstyle \C{2}{q} + \C{7}{q} + \C{14}{q} & 
      \scriptstyle \text{- -} \\
 15 & \frac{2+q}{1+q+q^2}+\frac{4+3 q+2 q^2+q^3}{1+q+q^2+q^3+q^4}+\frac{8-7 q+5 q^3-4 q^4+3 q^5-q^7}{1-q+q^3-q^4+q^5-q^7+q^8} & 
      \scriptstyle \C{3}{q} + \C{5}{q} + \C{15}{q} & 
      \scriptstyle \text{- -} \\
 16 & \frac{1}{1+q}+\frac{2}{1+q^2}+\frac{4}{1+q^4}+\frac{8}{1+q^8} & 
      \scriptstyle \C{2}{q} + \C{4}{q} + \C{8}{q} + \C{16}{q} & 
      \scriptstyle \C{2}{q} + 2\C{2}{q^2} + 4\C{2}{q^4} + 8\C{2}{q^8} \\
     \hline 
\end{array}$
}
\end{equation*} 

\bigskip

\caption{\textbf{Expansions of Lambert Series Terms by Cyclotomic Polynomial Primitives.} 
         The double dashes (-{}-) in the rightmost column of the table indicate that the 
          entry is the same as the previous column to distinguish between the cases where we 
          apply our special reduction formulas.}
\label{table_qnOver1mqn_LambertSeriesTerm_exps} 
\vspace*{-\baselineskip} 

\end{table} 

\subsection{Factored Lambert series expansions} 

To provide some intuition to the factorizations of the terms in our 
Lambert series generating functions defined above, the listings in 
Table \ref{table_qnOver1mqn_LambertSeriesTerm_exps} provide the first 
several expansions of the right-hand-sides of the next equations 
according to the optimal applications of 
\eqref{eqn_CyclotomicPolysPrimeProperties} in our new formulas. 
The components highlighted by the examples in the table 
form the key terms of our new exact formula expansions. 
Notably, we see that we may write the expansions of the individual 
Lambert series terms as
\begin{align*} 
\frac{n q^n}{1-q^n} + n - \frac{1}{1-q} & = \sum_{\substack{d|n \\ d > 1}} 
     \widetilde{\Phi}_d(q), 
\end{align*} 
where we can reduce the index orders of the cyclotomic polynomials, $\Phi_n(q)$, and 
their logarithmic derivatives, $\widetilde{\Phi}_d(q)$, in lower-indexed 
cyclotomic polynomials with $q$ transformed into powers of $q$ to powers of primes 
according to the identities noted in \eqref{eqn_CyclotomicPolysPrimeProperties} 
\citep[\cf \S 3]{DAVENPORT-MULTNUMT} \cite[\cf \S 13.2]{CYCLOTOMIC-POLYS}. 
An appeal to the logarithmic derivative of a product of differentiable rational 
functions and the definition given in \eqref{eqn_Pin_Cyclotomic2_defs} of the 
last definition, allows us to prove that for each natural number 
$n \geq 2$ we have that
\begin{align} 
\label{eqn_Intro_LambertSeriesTerm_exps_v1} 
\frac{q^n}{1-q^n} & = -1 + \frac{1}{n(1-q)} + \frac{1}{n} \sum_{\substack{d|n \\ d > 1}} 
     \widetilde{\Phi}_d(q). 
\end{align} 

\begin{remark}[Experimental Intuition for These Formulas]
We begin by pointing our that the genesis of the formulas proved in 
Section \ref{Section_Proofs} (stated precisely below) came about by 
experimentally observing the exact polynomial expansions of the key 
Lambert series terms $\frac{q^n}{1-q^n}$ which provide generating functions for 
$\sigma_{\alpha}(n)$ with \emph{Mathematica}. Namely, the computer algebra routines 
employed by default in \emph{Mathematica} are able to produce the semi-factored output 
reproduced in Table \ref{table_qnOver1mqn_LambertSeriesTerm_exps}. 
Without this computationally driven means for motivating experimental mathematics, we 
would most likely never have noticed these subtle formulas for the often-studied 
class of sum-of-divisors functions! 

The motivation for our definition of the next three divisor sum variants 
given in Definition \ref{def_Notation_for_CompDivSums_Sinq} is to effectively 
exploit the particularly desirable properties of the coefficients of these polynomial expansions 
when they are treated as generating functions for the generalized sum-of-divisors functions in the 
main results stated in the next section. More to the point, when a natural number $d$ is of the form 
$d = p^k, 2p^k$ for some prime $p$ and exponent $k \in \mathbb{Z}^{+}$, we have the reduction formulas 
cited in \eqref{eqn_CyclotomicPolysPrimeProperties} above to translate the implicit forms of the 
cyclotomic polynomials $\Phi_d(q)$ into 
polynomials in now powers of $q^{p^k}$ indexed only by the sums over primes $p$. 

The third and fourth columns of Table \ref{table_qnOver1mqn_LambertSeriesTerm_exps} 
naturally suggest by computation the 
exact forms of the (logarithmic derivative) polynomial expansions we are looking for 
to expand our Lambert series terms. 
In effect, the observation of these trends in the polynomial expansions of $1-q^n$ 
led to the intuition motivating our new results within this article. 
In particular, we introduce the notation in the next definition corresponding to 
component sums employed to express sums over the previous identity in our key 
results stated in the next pages of the article. 
\end{remark}

\begin{definition}[Notation for component divisor sums] 
\label{def_Notation_for_CompDivSums_Sinq} 
For fixed $q$ and any $n \geq 1$, we define the component sums, $\widetilde{S}_{i,n}(q)$ for 
$i = 0,1,2$ as follows: 
\begin{align*} 
\widetilde{S}_{0,n}(q) & = 
     \sum_{\substack{d|n \\ d > 1 \\ d \neq p^k, 2p^k}} \widetilde{\Phi}_d(q) \\ 
\widetilde{S}_{1,n}(q) & = \sum_{p|n} \Pi_{p^{\nu_p(n)}}(q) \\ 
\widetilde{S}_{2,n}(q) & =  \sum_{\substack{2p|n \\ p > 2}}\Pi_{p^{\nu_p(n)}}(-q). 
\end{align*} 
\end{definition} 

\subsection{Statements of key results and characterizations} 

Recall that for any $\alpha \in \mathbb{C}$ the generalized sum-of-divisors function 
is defined by the divisor sum 
\[
\sigma_{\alpha}(x) = \sum_{d|x} d^{\alpha}. 
\]
We use the following notation for the generalized $\alpha$-order harmonic number 
sequences: 
\[
H_n^{(\alpha)} := \sum_{k=1}^n k^{-\alpha}. 
\]

\begin{prop}[Series coefficients of the component sums] 
\label{prop_ComponentSum_SeriesCoeffs} 
For any fixed $\alpha \in \mathbb{C}$ and integers $x \geq 1$, we have the 
following components of the partial sums of the Lambert series generating 
cited next in Theorem \ref{theorem_ExactFormulas_v2}: 
\begin{align*} 
\tag{i} 
\widehat{S}_0^{(\alpha)}(x) & := 
     [q^x] \sum_{n=1}^x \widetilde{S}_{0,n}(q) n^{\alpha-1} =: 
     \tau_{\alpha}(x) \\ 
\tag{ii} 
\widehat{S}_1^{(\alpha)}(x) & := 
[q^x] \sum_{n=1}^x \widetilde{S}_{1,n}(q) n^{\alpha-1} = 
     \sum_{p \leq x} \sum_{k=1}^{\nu_p(x)+1} 
     p^{\alpha k-1} H_{\left\lfloor \frac{x}{p^k} \right\rfloor}^{(1-\alpha)} 
     \left( 
     p\left\lfloor \frac{x}{p^k} \right\rfloor -p \left\lfloor 
     \frac{x}{p^k} - \frac{1}{p} \right\rfloor - 1 
     \right) \\ 
\notag
\widehat{S}_2^{(\alpha)}(x) & := 
[q^x] \sum_{n=1}^x \widetilde{S}_{2,n}(q) n^{\alpha-1} \\ 
\tag{iii} 
     & \phantom{:} = 
     \sum_{3 \leq p \leq x} \sum_{k=1}^{\nu_p(x)+1} 
     \frac{p^{\alpha k-1}}{2^{1-\alpha}} 
     H_{\left\lfloor \frac{x}{2 p^k} \right\rfloor}^{(1-\alpha)} 
     (-1)^{\left\lfloor \frac{x}{p^{k-1}} \right\rfloor} 
     \left( 
     p\left\lfloor \frac{x}{p^k} \right\rfloor -p \left\lfloor 
     \frac{x}{p^k} - \frac{1}{p} \right\rfloor - 1 
     \right). 
\end{align*} 
\end{prop} 

The precise form of the expansions in (i) of the previous proposition and its connections to the 
Ramanujan sums, $c_q(n)$, is explored in the results stated in 
Proposition \ref{remark_RamanujansSum} of the next section. 

\begin{theorem}[Exact formulas for the generalized sum-of-divisors functions] 
\label{theorem_ExactFormulas_v2} 
\label{theorem_main_SOD_result} 
For any fixed $\alpha \in \mathbb{C}$ and natural numbers $x \geq 1$, we have the 
following generating function formula: 
\begin{align*} 
\sigma_{\alpha}(x) & = H_x^{(1-\alpha)} + 
     \widehat{S}_0^{(\alpha)}(x) + \widehat{S}_1^{(\alpha)}(x) + \widehat{S}_2^{(\alpha)}(x). 
\end{align*} 
\end{theorem}  

\subsection{Remarks} 

We first have a few 
remarks about symmetry in the identity from the theorem in the context of 
negative-order divisor functions of the form 
\[
\sigma_{-\alpha}(x) = \sum_{d|x} \left(\frac{x}{d}\right)^{-\alpha} = 
     \frac{\sigma_{\alpha}(x)}{x^{\alpha}},\ \alpha \geq 0, 
\]
and a brief overview of the applications we 
feature in Section \ref{Section_Applications}. 

\subsubsection{Symmetric forms of the exact formulas} 

For integers $\alpha \in \mathbb{N}$, we can express the ``negative-order'' 
harmonic numbers, $H_n^{(-\alpha)}$, in terms of the 
\emph{generalized Bernoulli numbers} (polynomials) as \emph{Faulhaber's formula}
\begin{align*} 
\sum_{m=1}^{n} m^{\alpha} & = \frac{1}{\alpha+1}\left( 
     B_{\alpha+1}(n+1) - B_{\alpha+1} 
     \right) \\ 
     & = \frac{1}{(\alpha + 1)} \sum_{j=0}^{\alpha} \binom{\alpha+1}{j} B_j n^{\alpha+1-j}.
\end{align*} 
Then since the convolution formula above proves that 
$\sigma_{-\beta}(n) = \sigma_{\beta}(n) / n^{\beta}$ whenever $\beta > 0$, 
we may also expand the right-hand-side of the theorem in the symmetric form of 
\begin{align} 
\notag 
\sigma_{\alpha}(x) & = x^{\alpha} \left(H_x^{(\alpha+1)} + 
     \tau_{-\alpha}(x) + 
     \widehat{S}_1^{(-\alpha)}(x) + \widehat{S}_2^{(-\alpha)}(x) 
     \right), 
\end{align} 
when $\alpha > 0$ is strictly real-valued. 
In particular, we are able to restate Proposition \ref{prop_ComponentSum_SeriesCoeffs} and 
Theorem \ref{theorem_main_SOD_result} together in the following alternate unified form 
where $c_d(x)$ denotes a \emph{Ramanujan sum} (see 
Proposition \ref{remark_RamanujansSum}): 

\begin{theorem}[Symmetric Forms of the Exact Formulas]
\label{theorem_main_SOD_result_V2_symmetric form} 
For any fixed $\alpha \in \mathbb{C}$ and integers $x \geq 1$, we have the 
following formulas: 
\begin{align*} 
\tag{i} 
\widehat{S}_0^{(-\alpha)}(x) & = 
     \sum_{d=1}^{x} H_{\floor{\frac{x}{d}}}^{(\alpha+1)} \cdot \frac{c_d(x)}{d^{\alpha+1}} 
     \cdot \chi_{\PP}(d) \\ 
\tag{ii} 
\widehat{S}_1^{(-\alpha)}(x) & = 
     \sum_{\substack{p \leq x \\ p\mathrm{\ prime}}} \left[ 
     \sum_{k=1}^{\nu_p(x)} \frac{(p-1)}{p^{\alpha k+1}} H_{\frac{x}{p^k}}^{(\alpha+1)} - 
     \frac{1}{p^{\alpha \cdot \nu_p(x) + \alpha + 1}} 
     H_{\floor{\frac{x}{p^{\nu_p(x)+1}}}}^{(\alpha+1)} 
     \right] \\ 
\tag{iii} 
\widehat{S}_2^{(-\alpha)}(x) & = 
     \frac{(-1)^{x}}{2^{\alpha+1}} \sum_{\substack{p \leq x \\ p\mathrm{\ prime}}} \left[ 
     \sum_{k=1}^{\nu_p(x)} \frac{(p-1)}{p^{\alpha k+1}} H_{\floor{\frac{x}{2 p^k}}}^{(\alpha+1)} - 
     \frac{1}{p^{\alpha \cdot \nu_p(x) + \alpha + 1}} 
     H_{\floor{\frac{x}{2 p^{\nu_p(x)+1}}}}^{(\alpha+1)} 
     \right]. 
\end{align*} 
The generalized sum of divisors functions then have the following explicit expansions involving these 
formulas as 
\begin{equation}
\label{eqn_SigmaAlphax_Negative-Order_Symmetry_Ident} 
\sigma_{\alpha}(x) = x^{\alpha} \left(H_x^{(\alpha+1)} + 
     \widehat{S}_0^{(-\alpha)}(x) + 
     \widehat{S}_1^{(-\alpha)}(x) + \widehat{S}_2^{(-\alpha)}(x) 
     \right). 
\end{equation}
\end{theorem}
We notice that this symmetry identity given in 
Theorem \ref{theorem_main_SOD_result_V2_symmetric form} 
provides a curious, and necessarily deep, relation between the 
Bernoulli numbers and the partial sums of the Riemann zeta function 
involving nested sums over the primes. 
It also leads to a direct proof of the known asymptotic results for the 
summatory functions \cite[\S 27.11]{NISTHB} 
\[
\sum_{n \leq x} \sigma_{\alpha}(n) = \frac{\zeta(\alpha+1)}{\alpha+1} x^{\alpha+1} + 
     O\left(x^{\max(1, \alpha)}\right),\ \alpha > 0, \alpha \neq 1. 
\]
We will explore this direct proof based on 
Theorem \ref{theorem_main_SOD_result_V2_symmetric form} 
in more detail as an application given in 
Section \ref{subSection_Apps_Asymptotics}. 

\section{Proofs of our new results} 
\label{Section_Proofs} 

\subsection{Motivating the proof of the new formulas} 
\label{subSection_ProofMotivations_RamSums}  

\begin{example}
We first revisit a computational 
example of the rational functions defined by the logarithmic derivatives in 
Definition \ref{def_NotationProductsLogDer} from 
Table \ref{table_qnOver1mqn_LambertSeriesTerm_exps}. 
We make use of the next variant of the identity in 
\eqref{eqn_KeyCyclotomicPolyProdIdent_ForLLambdaFn} in the proof below 
which is obtained by M\"oebius inversion:  
\begin{align} 
\label{eqn_CyclotomicFactorization_stmt_v2} 
\Phi_n(q) = \prod_{d|n} (q^d-1)^{\mu(n/d)}. 
\end{align} 
In the case of our modified rational cyclotomic polynomial functions, 
$\widetilde{\Phi}_n(q)$, when $n := 15$, we use this product to expand the 
definition of the function as 
\begin{align*} 
\widetilde{\Phi}_{15}(q) & = 
\frac{1}{x} \cdot \frac{d}{dq}\left[\log\left(\frac{(1-q^3)(1-q^5)}{(1-q)(1-q^{15})} 
     \right)\right]\Biggr|_{q \rightarrow 1/q} \\ 
     & = 
     \frac{3}{1-q^3} + \frac{5}{1-q^5} - \frac{1}{1-q} - \frac{15}{1-q^{15}} \\ 
     & = 
     \frac{8-7q+5q^3-4q^4+3q^5-q^7}{1-q+q^3-q^4+q^5-q^7+q^8}. 
\end{align*} 
The procedure for transforming the difficult-looking terms involving the 
cyclotomic polynomials when the Lambert series terms, $q^n / (1-q^n)$, are 
expanded in partial fractions as in 
Table \ref{table_qnOver1mqn_LambertSeriesTerm_exps} 
is essentially the same as this example for the 
cases we will encounter here. 
In general, we have the next simple lemma when $n$ is a positive integer. 
\end{example} 

\begin{lemma}[Key characterizations of the tau divisor sums] 
\label{lemma_PhiTildend_MudOverTerm_DivSum_exp} 
For integers $n \geq 1$ and any indeterminate $q$, we have the following 
expansion of the functions in \eqref{eqn_Pin_Cyclotomic2_defs}: 
\begin{align*} 
\widetilde{\Phi}_n(q) & = \sum_{d|n} \frac{d \cdot \mu(n/d)}{(1-q^d)}. 
\end{align*} 
In particular, we have that 
\begin{align*} 
\widetilde{S}_{0,n}(q) & = \sum_{d|n} \sum_{r|d} 
     \frac{r \cdot \widetilde{\chi}_{\PP}(d) \cdot \mu(d/r)}{(1-q^r)}. 
\end{align*} 
\end{lemma} 
\begin{proof} 
The proof is essentially the same as the example given above. 
Since we can refer to 
this illustrative example, we only need to sketch the details to the 
remainder of the proof. In particular, we notice that since we have the known 
identity for the cyclotomic polynomials given by 
\begin{align} 
\notag 
\Phi_n(x) & = \prod_{d|n} (1-q^d)^{\mu(n/d)} 
\end{align} 
we can take logarithmic derivatives to obtain that 
\begin{align*} 
\frac{1}{x} \cdot \frac{d}{dq}\left[\log\left(1-q^d\right)^{\pm 1}\right] 
     \Biggr|_{q \rightarrow 1/q} & = 
     \mp \frac{d}{q^d\left(1-\frac{1}{q^d}\right)} = 
     \pm \frac{d}{1-q^d}, 
\end{align*} 
which applied inductively leads us to our result. 
\end{proof} 

\begin{prop}[Connections to Ramanujan sums] 
\label{remark_RamanujansSum} 
Let the following notation denote a shorthand for the divisor sum terms in 
Theorem \ref{theorem_main_SOD_result}: 
\[ 
\tau_{\alpha}(x) := \widehat{S}_0^{(\alpha)}(x) = 
     [q^x] \sum_{n=1}^x \widetilde{S}_{0,n}(q) n^{\alpha-1}. 
\] 
We have the following two characterizations of the functions $\tau_{\alpha}(x)$ 
expanded in terms of Ramanujan's sum, $c_q(n)$, where $\mu(n)$ denotes the M\"obius function and 
$\varphi(n)$ is Euler's totient function: 
\begin{align*} 
\tau_{\alpha+1}(x) & = \sum_{\substack{d=1 \\ d \neq p^k,2p^k}}^x 
     H_{\floor{\frac{x}{d}}}^{(-\alpha)} \cdot d^{\alpha} \cdot c_d(x) \\ 
     & = \sum_{\substack{d=1 \\ d \neq p^k,2p^k}}^x 
     H_{\floor{\frac{x}{d}}}^{(-\alpha)} \cdot d^{\alpha} \cdot 
     \mu\left(\frac{d}{(d, x)}\right) 
     \frac{\varphi(d)}{\varphi\left(\frac{d}{(d,x)}\right)}. 
\end{align*} 
\end{prop} 
\begin{proof} 
First, we observe that the contribution of the first (zero-indexed) sums in 
Theorem \ref{theorem_ExactFormulas_v2} correspond to the coefficients 
\begin{align*} 
\tau_{\alpha+1}(x) 
     & = 
     [q^x]\left(\sum_{k=1}^x \sum_{\substack{d|k \\ d \neq p^s,2p^s}} \sum_{r|d} 
     \frac{r \cdot \mu(d/r)}{(1-q^r)} k^{\alpha}\right) \\ 
\notag 
     & = 
     \sum_{k=1}^x \sum_{r|x} \sum_{\substack{d|k \\ d \neq p^s,2p^s}} r \cdot \mu(d/r) 
     \cdot \Iverson{r|d} \cdot k^{\alpha} \\ 
\label{eqn_RamSumIdent_ref} 
     & = 
     \sum_{k=1}^x \sum_{\substack{d|k \\ d \neq p^s,2p^s}} \sum_{r|(d,x)} r \cdot \mu(d/r) 
     \cdot k^{\alpha}. 
\end{align*} 
Then since we can easily prove the identity that 
\[
\sum_{k=1}^{x} \sum_{d|k} f(d) g(k/d) = 
     \sum_{d=1}^{x} f(d) \left(\sum_{k=1}^{\floor{\frac{x}{d}}} g(k)\right), 
\]
for any prescribed arithmetic functions $f$ and $g$, 
we can also expand the right-hand-side of the previous equation as 
\begin{equation} 
\label{eqn_tauAlphaFunc_KeyExp} 
\tau_{\alpha+1}(x) = \sum_{\substack{d=1 \\ d \neq p^k,2p^k}}^x 
     \left(\sum_{r|(d,x)} r \mu(d/r)\right) 
     H_{\floor{\frac{x}{d}}}^{(-\alpha)} \cdot d^{\alpha}. 
\end{equation} 
Thus the identities stated in the proposition 
follow by expanding out Ramanujan's sum in the form of 
\citep[\S 27.10]{NISTHB} \citep[\S A.7]{ADDITIVENUMT} \citep[\cf \S 5.6]{HARDYWRIGHTNUMT} 
\[
c_q(n) = \sum_{d|(q,n)} d \cdot \mu(q/d). 
     \qedhere
\] 
\end{proof} 

\begin{remark}
Ramanujan's sum also satisfies the convenient bound that $|c_q(n)| \leq (n, q)$ for all $n,q \geq 1$, 
which can be used to obtain asymptotic estimates in the form of 
upper bounds for these sums when $q$ is not prime or a prime power. 
Additionally, it is related to periodic exponential sums (modulo $k$) of the more general forms 
\[
s_k(n) = \sum_{d|(n,k)} f(d) g\left(\frac{k}{d}\right), 
\]
where $s_k(n)$ has the finite Fourier series expansion 
\[
s_k(n) = \sum_{m=1}^{k} a_k(m) e^{2\pi\imath n/k}, 
\]
with coefficients given by the divisor sums \citep[\S 27.10]{NISTHB} 
\[
a_k(m) = \sum_{d|(m,k)} g(d) f\left(\frac{k}{d}\right) \frac{d}{k}. 
\]
It turns out that the terms in the formulas for $\sigma_{\alpha}(x)$ represented by these 
sums, $\tau_{\alpha}(x)$ provide detailed insight into the error estimates 
for the summatory functions over the generalized sum-of-divisors functions. 
We computationally investigate the properties of the new expansions we can obtain for 
these sums using the new formulas from the theorem as applications in 
Section \ref{Section_Applications}. 
\end{remark} 

\subsection{Proofs of the theorems} 

\begin{proof}[Proof of Theorem \ref{theorem_ExactFormulas_v2}] 
We begin with a well-known divisor product formula involving the cyclotomic polynomials 
when $n \geq 1$ and $q$ is fixed: 
\[
q^n-1 = \prod_{d|n} \Phi_d(q). 
\]
Then by logarithmic differentiation we can see that 
\begin{align} 
\label{eqn_proof_tag_eqn_i-v1} 
\frac{q^n}{1-q^n} & = -1 + \frac{1}{n(1-q)} + \frac{1}{n} \sum_{\substack{d|n \\ d > 1}} 
     \widetilde{\Phi}_d(q) \\ 
\notag 
     & = -1 + \frac{1}{n(1-q)} + \frac{1}{n}\left( 
     \widetilde{S}_{0,n}(q) + \widetilde{S}_{1,n}(q) + 
     \widetilde{S}_{2,n}(q)\right). 
\end{align} 
The last equation is obtained from the first expansion in \eqref{eqn_proof_tag_eqn_i-v1} 
above by identifying the next two sums as 
\[
\Pi_n(q) = \sum_{\substack{d|n \\ d>1}} \widetilde{\Phi}_n(1/q) = 
     \sum_{j=0}^{n-2} \frac{(n-1-j) q^j (1-q)}{1-q^n}. 
\]
Here we are implicitly using the known expansions of the cyclotomic polynomials 
which condense the order $n$ of the polynomials by exponentiation of the 
indeterminate $q$ when $n$ contains a factor of a prime power given by 
\eqref{eqn_CyclotomicPolysPrimeProperties} in the introduction. 
Finally, we complete the proof by summing the right-hand-side of 
\eqref{eqn_proof_tag_eqn_i-v1} over $n \leq x$ times the weight $n^{\alpha}$ to 
obtain the $x^{th}$ partial sum of the Lambert series generating function for 
$\sigma_{\alpha}(x)$ \citep[\S 17.10]{HARDYWRIGHTNUMT} \citep[\S 27.7]{NISTHB}, 
which since each term in the summation contains a power of $q^n$ is $(x+1)$-order 
accurate to the terms in the infinite series. 
\end{proof} 

\begin{proof}[Proof of Proposition \ref{prop_ComponentSum_SeriesCoeffs}] 
The identity in (i) follows from Lemma \ref{lemma_PhiTildend_MudOverTerm_DivSum_exp}. 
Since $\Phi_{2p}(q) = \Phi_p(-q)$ for any prime $p$, we are essentially in the 
same case with the two component sums in (ii) and (iii). We outline the proof of our 
expansion for the first sum, $\widetilde{S}_{1,n}(q)$, and note the small 
changes necessary along the way to adapt the proof to the second sum case. 
By the properties of the cyclotomic polynomials expanded in 
\eqref{eqn_CyclotomicPolysPrimeProperties}, we may 
factor the denominators of $\Pi_{p^{\varepsilon_p(n)}}(q)$ into smaller 
irreducible factors of the same polynomial, $\Phi_p(q)$, with inputs varying 
as special prime-power powers of $q$. More precisely, we may expand 
\[
Q_{p,k}^{(n)}(q) := \frac{\sum\limits_{j=0}^{p-2} (p-1-j) q^{p^{k-1}j}}{ 
     \sum\limits_{i=0}^{p-1} q^{p^{k-1}i}}, 
\] 
and 
\begin{align*} 
\widetilde{S}_{1,n}(q) & = \sum_{p \leq n} \sum_{k=1}^{\nu_p(n)} 
     Q_{p,k}^{(n)}(q) \cdot p^{k-1}. 
\end{align*} 
In performing the sum $\sum_{n \leq x} Q_{p,k}^{(n)}(q) p^{k-1} n^{\alpha-1}$, 
these terms of the $Q_{p,k}^{(n)}(q)$ occur again, or have a repeat 
coefficient, every $p^k$ terms, so we form the coefficient sums for these terms as 
\begin{equation*} 
\sum_{i=i}^{\left\lfloor \frac{x}{p^k} \right\rfloor} \left(i p^k\right)^{ 
     \alpha-1} \cdot p^{k-1} = p^{k\alpha-1} \cdot 
     H_{\left\lfloor \frac{x}{p^k} \right\rfloor}^{(1-\alpha)}. 
\end{equation*} 
We can also compute the inner sums in the previous equations exactly for any 
fixed $t$ as 
\[
     \sum_{j=0}^{p-2} (p-1-j) t^j = \frac{(p-1)-pt-t^p}{(1-t)^2}, 
\]
where the corresponding paired denominator sums in these terms are given by 
$1+t+t^2+\cdots+t^{p-1} = (1-t^p) / (1-t)$. 
We now assemble the full sum over $n \leq x$ we are after in this proof as 
\begin{align*} 
\sum_{n \leq x} \widetilde{S}_{1,n}(q) \cdot n^{\alpha-1} & = 
     \sum_{p \leq x} \sum_{k=1}^{\varepsilon_p(x)} p^{k\alpha-1} 
     H_{\left\lfloor \frac{x}{p^k} \right\rfloor}^{(1-\alpha)}
     \frac{(p-1)-p q^{p^{k-1}} + q^{p^k}}{(1-q^{p^{k-1}})(1-q^{p^k})}. 
\end{align*} 
The corresponding result for the second sums is obtained similarly with the 
exception of sign changes on the coefficients of the powers of $q$ in the last 
expansion. 

We compute the series coefficients of one of the three cases in the previous 
equation to show our method of obtaining the full formula. 
In particular, the right-most term in these expansions leads to the double sum 
\begin{align*} 
C_{3,x,p} & := [q^x] \frac{q^{p^k}}{(1\mp q^{p^{k-1}})(1\mp q^{p^k})} \\ 
     & \phantom{:} = [q^x] \sum_{n,j \geq 0} (\pm 1)^{n+j} 
     q^{p^{k-1}(n+p+jp)}. 
\end{align*} 
Thus we must have that $p^{k-1} | x$ in order to have a non-zero coefficient and 
for $n := x / p^{k-1}-jp-p$ with $0 \leq j \leq x/p^k-1$ we can compute these 
coefficients explicitly as 
\begin{align*} 
C_{3,x,p} & := (\pm 1)^{\lfloor x / p^{k-1} \rfloor} \times \sum_{j=0}^{ 
     \lfloor x/p^k-1 \rfloor} 1 = 
     (\pm 1)^{\lfloor x / p^{k-1} \rfloor} 
     \left\lfloor \frac{x}{p^k}-1 \right\rfloor + 1 = 
     (\pm 1)^{\lfloor x / p^{k-1} \rfloor} 
     \left\lfloor \frac{x}{p^k} \right\rfloor. 
\end{align*} 
With minimal simplifications we have arrived at our claimed result 
in the proposition. The other two similar computations follow similarly. 
\end{proof} 

\begin{proof}[Proof of Theorem \ref{theorem_main_SOD_result_V2_symmetric form}] 
We first note that since for non-negative $\alpha \geq 0$, we have 
\[
\sigma_{-\alpha}(x) = \sum_{d|x} d^{-\alpha} = 
     \sum_{d|x} \left(\frac{x}{d}\right)^{-\alpha} = 
     \frac{\sigma_{\alpha}(x)}{x^{\alpha}}, 
\]
we can see that the formula in \eqref{eqn_SigmaAlphax_Negative-Order_Symmetry_Ident}
follows immediately from Theorem \ref{theorem_main_SOD_result}. It remains to 
prove the subformulas in (i)--(iii) of the theorem. The first formula for 
$S_0^{(-\alpha)}(x)$ corresponds to the formulas we derived in 
Proposition \ref{remark_RamanujansSum} of the previous subsection for these cases of 
negative-order $\alpha$. The second two formulas follow from 
Proposition \ref{prop_ComponentSum_SeriesCoeffs} by expanding the 
cases of the floor function inputs according to the inner 
index $k$ in the ranges $k \in [1, \nu_p(x)]$, i.e., where $x / p^k \in \mathbb{Z}^{+}$, 
and then the remainder index case of $k := \nu_p(x) + 1$. 
\end{proof} 

\section{Applications of the new formulas} 
\label{Section_Applications} 

\subsection{Asymptotics of sums of the divisor functions} 
\label{subSection_Apps_Asymptotics} 

We can use the new exact formula proved by 
Theorem \ref{theorem_ExactFormulas_v2} to asymptotically estimate 
partial sums, or \emph{average orders} of the respective arithmetic functions, 
of the next form for integers $x \geq 1$:
\begin{align} 
\label{eqn_def_SummatoryFunc_SigmaAlphaBetaX} 
\Sigma^{(\alpha, \beta)}(x) & := \sum_{n \leq x} \frac{\sigma_{\alpha}(n)}{n^{\beta}}. 
\end{align} 
We similarly define $\Sigma_{\alpha}(x) := \Sigma^{(\alpha,0)}(x)$ in this notation. 
In the special cases where $\alpha := 0, 1$, we restate a few more famous 
formulas providing well-known classically (and newer) established 
asymptotic bounds for sums of this form 
as follows where $\gamma \approx 0.577216$ is Euler's gamma constant, 
$d(n) \equiv \sigma_0(n)$ denotes the (Dirichlet) \emph{divisor function}, and 
$\sigma(n) \equiv \sigma_1(n)$ the (ordinary) \emph{sum-of-divisors function} 
\citep{HUXLEY2003,CARELLA-EXPLFORMULA-DIVFN,WALFISZ63} \citep[\cf \S 27.11]{NISTHB}: 
\begin{align} 
\label{eqn_DivFn_WellKnown_AsympFormulas}
\Sigma_0(x) & := \sum_{n \leq x} d(n) = 
     x \log x + (2\gamma-1) x + O\left(x^{\frac{131}{416}}\right) \\ 
\notag 
\Sigma^{(0,1)}(x) & := \sum_{n \leq x} \frac{d(n)}{n} = 
     \frac{1}{2} (\log x)^2 + 2\gamma \log x + 
     O\left(x^{-2/3}\right) \\ 
\notag 
\Sigma_1(x) & := \sum_{n \leq x} \sigma(n) = \frac{\pi^2}{12} x^2 + O(x \log^{2/3} x).
\end{align} 
For the most part, we suggest tackling potential 
improvements to these possible asymptotic formulas 
through our new results given in the theorem and in the symmetric identity 
\eqref{eqn_SigmaAlphax_Negative-Order_Symmetry_Ident} as a 
topic for exploration and subsequent numerical verification using the 
computational tools at our disposal. 

\begin{example}[Average order of the divisor function] 
For comparison with the leading terms in the first of the previous 
expansions, we can prove the next formula using summation by parts for 
integers $r \geq 1$. 
\begin{equation*} 
\sum_{j=1}^n H_j^{(r)} = (n+1) H_n^{(r)} - H_n^{(r-1)} 
\end{equation*} 
Then using inexact approximations for the summation terms in the theorem, we are 
able to evaluate the leading non-error term in the following sum for large 
integers $t \geq 2$ since $H_n^{(1)} \sim \log n + \gamma$: 
\begin{align*} 
\Sigma_t^{(0, 0)} & = -t + (t+1) H_t^{(1)} + O(t \cdot \log^3(t)) \\ 
     & \sim 
     (t+1) \log t + (\gamma-1) t + \gamma + O(t \cdot \log^3(t)). 
\end{align*} 
It is similarly not difficult to obtain a related estimate for the second 
famous divisor sum, $\Sigma_t^{(0, 1)}$, using the symmetric identity in 
\eqref{eqn_SigmaAlphax_Negative-Order_Symmetry_Ident} of the introduction. 
\end{example} 

There is an obvious finite sum identity which 
generates partial sums of the generalized sum-of-divisors functions in the 
following forms \citep[\cf \S 7]{CARELLA-EXPLFORMULA-DIVFN} 
\citep{ONCIRCLE-AND-DIVISOR}: 
\begin{align} 
\label{eqn_SigmaAlphax_formulas_v1} 
\Sigma_{\alpha}(x) & := \sum_{n \leq x} \sigma_{\alpha}(n) = 
     \sum_{d \leq x} \left\lfloor \frac{x}{d} \right\rfloor \cdot d^{\alpha} \\ 
\notag 
     & \phantom{:} = 
     \sum_{m=0}^{\left\lfloor \frac{\log x}{\log 2} \right\rfloor} 
     \sum_{d=1}^{\left\lfloor \frac{x}{2^m} \right\rfloor - 
     \left\lfloor \frac{x}{2^{m+1}} \right\rfloor} 
     \left\lfloor \frac{x}{d+\left\lfloor x 2^{-(m+1)}\right\rfloor} \right\rfloor 
     \left(d+\left\lfloor \frac{x}{2^{m+1}} \right\rfloor \right)^k \\ 
\notag
     & \phantom{:} = 
     \sum_{m=0}^{\left\lfloor \frac{\log x}{\log p} \right\rfloor} 
     \sum_{d=1}^{\left\lfloor \frac{x}{p^m} \right\rfloor - 
     \left\lfloor \frac{x}{p^{m+1}} \right\rfloor} 
     \left\lfloor \frac{x}{d+\left\lfloor x p^{-(m+1)}\right\rfloor} \right\rfloor 
     \left(d+\left\lfloor \frac{x}{p^{m+1}} \right\rfloor \right)^k,\ 
     p \in \mathbb{Z}, p \geq 2. 
\end{align} 
Since the sums $\sum_{d=1}^m (d+a)^k$ are readily expanded by the Bernoulli 
polynomials, we may approach summing the last finite sum identity by parts 
\citep[\S 24.4(iii)]{NISTHB}. 
The relation of sums of this type corresponding to the divisor function case 
where $\alpha := 0$ are considered in the context of the Dirichlet divisor problem 
in \citep{ONCIRCLE-AND-DIVISOR} as are the evaluations of several sums 
involving the floor function such as we have in the statement of 
Theorem \ref{theorem_ExactFormulas_v2}. 
Similarly, we can arrive at additional exact identities for the average order sum 
variants defined above:  
\begin{align*} 
\Sigma_x^{(\alpha,\beta)} & = \sum_{d=1}^{x} d^{\alpha-\beta} \cdot H_{\floor{\frac{x}{d}}}^{(\beta)}. 
\end{align*} 
We can extend the known classical result for the sums $\Sigma_{\alpha}(x)$ given by 
\[
\Sigma_{\alpha}(x) := \sum_{n \leq x} \sigma_{\alpha}(x) = 
     \frac{\zeta(\alpha+1)}{\alpha+1} x^{\alpha+1} + 
     O\left(x^{\max(1, \alpha)}\right),\ \alpha > 0, \alpha \neq 1, 
\]
to the cases of these modified summatory functions using the new formulas proved in 
Theorem \ref{theorem_main_SOD_result_V2_symmetric form}. 
The next result provides the exact details of the 
limiting asymptotic relations for the sums $\Sigma^{(\alpha, \beta)}(x)$. 

\begin{theorem}[Asymptotics for Summatory Functions] 
\label{theorem_SigmaAlphaBeta_SumFunc_Asymptotics} 
For integers $\alpha > 1$ and $2 \leq \beta \leq \alpha$, 
we have that the summatory functions 
$\Sigma^{(\alpha, \beta)}(x)$ defined in 
\eqref{eqn_def_SummatoryFunc_SigmaAlphaBetaX} above 
satisfy the following asymptotic properties: 
\begin{align*}
\Sigma^{(\alpha, \beta)}(x) & = 
     \frac{\zeta(\alpha+1) x^{\alpha+1-\beta}}{(\alpha+1-\beta)} \left( 
     1 - C_1(\alpha) + C_{2,0}(\alpha) + C_3(\alpha) + 
     C_6(\beta) + C_{7,0}(\alpha)\right) \\ 
     & \phantom{=\ } + 
     \sum_{j=1}^{\alpha-\beta} \binom{\alpha+1-\beta}{j} 
     \frac{B_j x^{\alpha+1-\beta-j}}{\alpha+1-\beta} 
     \left(1 + C_{2,j}(\alpha) + C_{7,j}(\alpha)\right) + \sum_{j=0}^{\alpha-\beta} 
     C_{4,j}(\alpha, \beta) \zeta(\alpha+1) x^j \\ 
     & \phantom{=\ } + 
     \sum_{j=0}^{\alpha-\beta} \binom{\alpha-\beta}{j} \frac{C_5(\alpha, \beta) 
     (-1)^{\alpha-\beta-j} E_j}{2^{2\alpha+2-\beta}} + 
     O\left(\frac{x}{\log x}\right),  
\end{align*} 
where the absolute constants (depending only on $\alpha$ and $m$) are defined by 
\begin{align*} 
C_1(\alpha) & := \sum_{\substack{p \geq 2 \\ p\mathrm{\ prime}}} 
     \frac{(p-2)}{p(p-1)(p^{\alpha+1}-1)}\left[ 
     \frac{(p-1)}{p(p^{\alpha}-1)} - \frac{1}{p^{\alpha}}
     \right] \\ 
C_{2,m}(\alpha) & := \sum_{\substack{p \geq 2 \\ p\mathrm{\ prime}}} 
     \frac{(p-1)}{p^{\alpha+2-m} (p^{\alpha}-1)} \\ 
C_{3}(\alpha) & := \sum_{\substack{p \geq 3 \\ p\mathrm{\ prime}}} 
     \frac{(p-2)}{2^{\alpha+1} p(p-1)(p^{\alpha+1}-1)}\left[ 
     \frac{(p-1)}{p(p^{\alpha}-1)} - \frac{1}{p^{\alpha}}
     \right] \\ 
C_{4,m}(\alpha, \beta) & := \sum_{\substack{p \geq 2 \\ p\mathrm{\ prime}}} 
     \sum_{k=0}^{\alpha-\beta} \binom{\alpha-\beta}{k} \binom{\alpha-\beta-k}{m} 
     \frac{(-1)^{k+m} E_k \cdot (p-1)}{2^{2\alpha+2-\beta-m} p^{\beta+1+m} 
     (p^{\alpha}-1)} \\ 
C_5(\alpha, \beta) & := \sum_{\substack{p \geq 2 \\ p\mathrm{\ prime}}} 
     \frac{(p-1)}{p^{\beta+1} (p^{\alpha}-1)} \\ 
C_6(\beta) & := \sum_{\substack{p \geq 2 \\ p\mathrm{\ prime}}} 
     \frac{(p-2)}{p(p-1) (p^{\beta+1}-1)} \\ 
C_{7,m}(\alpha) & := -\sum_{\substack{p \geq 2 \\ p\mathrm{\ prime}}} 
     \frac{1}{p^{\alpha+1-m}}. 
\end{align*} 
In the previous equations, $B_n$ is a \emph{Bernoulli numbers} and $E_n$ denotes the 
\emph{Euler numbers}. 
\end{theorem} 
Due to the length of estimating some of the intricate nested sums from 
the formulas in Theorem \ref{theorem_main_SOD_result_V2_symmetric form}, 
we delay a complete proof of 
Theorem \ref{theorem_SigmaAlphaBeta_SumFunc_Asymptotics} 
to the last appendix section on page \pageref{page_AppendixA_Proofs}. 
Compared to the classical result related to these sums cited above, the error 
terms are of notably small order. This shows how accurate the new exact formulas 
can be in obtaining new asymptotic estimates. The expense of the small error 
terms is that the involved main terms have expanded in number and complexity. 

\subsection{A short application to perfect numbers} 
\label{subSection_PerfectNumbers} 

We turn our attention to an immediate application of our new results 
which is perhaps one of the most famous unresolved problems in number theory: 
that of determining the form of the perfect numbers. 
A \emph{perfect number} $p$ is a positive integer such that $\sigma(p) = 2p$. 
The first few perfect numbers are given by the sequence 
$\{ 6, 28, 496, 8128, 33550336, \ldots\}$. 
It currently is not known whether there are infinitely-many perfect numbers, or 
whether there exist odd perfect numbers. 
References to work on the distribution of the perfect number counting function, 
$V(x) := \#\{n\text{ perfect}: n \leq x\}$ are found in \citep[\S 2.7]{PRIMEREC}. 
Since we now have a fairly simple 
exact formula for the sum-of-divisors function, $\sigma(n)$, we briefly attempt to 
formulate conditions for an integer to be perfect within the scope of this article. 

It is well known that given a \emph{Mersenne prime} of the form $q = 2^p-1$ for 
some prime $p$, then we have corresponding perfect number of the form
$P = 2^{p-1}(2^p-1)$ 
\citep[\S 2.7]{PRIMEREC} \citep{LUSTIG-ALGINDEP-SUMDIVFNS}. 
We suppose that the positive integer $P$ has the form 
$P = 2^{p-1}(2^p-1)$ for some (prime) integer $p \geq 2$, and 
consider the expansion of the 
sum-of-divisors function on this input to our new exact formulas. 
Suppose that $R := 2^p-1 = r_1^{\gamma_1} r_2^{\gamma_2} \cdots r_k^{\gamma_k}$ 
is the prime factorization of this factor $R$ of $P$ where $\gcd(2, r_i) = 1$ 
for all $1 \leq i \leq k$ and that $R_s := R / s^{\nu_s(R)}$. 
Then by the formulas derived in 
Proposition \ref{prop_ComponentSum_SeriesCoeffs} we have by 
Theorem \ref{theorem_ExactFormulas_v2} that 
\begin{align*} 
\sigma(P) & = \frac{(p+1)}{2} P + \frac{P}{R} \left\lfloor \frac{R}{2} \right\rfloor 
     \left(2 \left\lfloor \frac{R}{2} \right\rfloor - 
     2 \left\lfloor \frac{R-1}{2} \right\rfloor - 1\right) + 
     \tau_1(P) \\ 
     & \phantom{=\frac{(p+1)}{2} P\ } + 
     \sum_{\substack{3 \leq s \leq P \\ s \text{ prime}}} 
     \frac{3}{2} 
     \frac{(s-1)}{s} P \cdot \nu_s(R) \\ 
     & \phantom{= \frac{(p-1)}{4} P\ } + 
     \sum_{\substack{3 \leq s \leq P \\ s \text{ prime}}} 
     s^{\nu_s(R)} \left(\left\lfloor \frac{2^{p-1} R_s}{s} \right\rfloor + 
     \left\lfloor \frac{2^{p-2}R_s}{s} \right\rfloor\right) 
     \times \\ 
     & \phantom{= \frac{(p-1)}{4} P\sum\sum\sum\ } \times 
     \left(s\left\lfloor \frac{2^{p-1} R_s}{s} \right\rfloor - 
     s \left\lfloor \frac{2^{p-1} R_s-1}{s} \right\rfloor -1\right). 
\end{align*} 
If we set $\sigma(P) = 2P$, i.e., construct ourselves a perfect number $P$ by 
assumption to work with, and then 
finally solve for the linear equation in $P$ from the last equation, 
we obtain that $P$ is perfect implies that the following 
condition holds: 
\begin{align} 
\label{eqn_PPerfectCond_v2} 
P & = -\frac{\tau_1(P) + 
     \sum\limits_{\substack{3 \leq s \leq P \\ s \text{ prime}}} 
     s^{\nu_s(R)} \left(\left\lfloor \frac{2^{p-1} R_s}{s} \right\rfloor + 
     \left\lfloor \frac{2^{p-2}R_s}{s} \right\rfloor\right) 
     \left(s\left\lfloor \frac{2^{p-1} R_s}{s} \right\rfloor - 
     s \left\lfloor \frac{2^{p-1} R_s-1}{s} \right\rfloor -1\right)}{ 
     \frac{(p-3)}{2} + \frac{1}{R} \left\lfloor \frac{R}{2} \right\rfloor 
     \left(2 \left\lfloor \frac{R}{2} \right\rfloor - 
     2 \left\lfloor \frac{R-1}{2} \right\rfloor - 1\right) + 
     \sum\limits_{\substack{3 \leq s \leq P \\ s \text{ prime}}} 
     \frac{3(s-1)}{2s} \cdot \nu_s(R)}. 
\end{align} 

\subsection{Corollaries and other identities for the generalized sum-of-divisors functions} 

\begin{cor}
We have the following two new 
noteworthy identities for $\sigma_{\alpha}(x)$ 
defined in terms of the famous Ramanujan sums and special multiplicative functions: 
\begin{align*} 
\sigma_{\alpha}(x) & = \sum_{d=1}^{x} H_{\floor{\frac{x}{d}}}^{(1-\alpha)} \cdot d^{\alpha-1} \cdot 
     c_d(x) \\ 
     & = \sum_{d=1}^{x} H_{\floor{\frac{x}{d}}}^{(1-\alpha)} \cdot d^{\alpha-1} \cdot 
     \mu\left(\frac{d}{(d, x)}\right) 
     \frac{\varphi(d)}{\varphi\left(\frac{d}{(d,x)}\right)}. 
\end{align*} 
\end{cor} 
\begin{proof} 
We can easily derive the 
following consequences from a modification of the 
result in \eqref{eqn_Intro_LambertSeriesTerm_exps_v1} from the introduction when 
$x \geq 1$ and for any $\alpha \in \mathbb{C}$: 
\begin{align*} 
\sigma_{\alpha}(x) & = [q^x]\left(\sum_{n=1}^x \sum_{d|n} \frac{1}{q} \cdot \frac{d}{dw}\Bigl[ 
     \log \Phi_d(w) \Bigr] \Bigr\rvert_{w=1/q} \times n^{\alpha-1}\right) \\ 
     & = [q^x]\left(\sum_{n=1}^x \sum_{d|n} \sum_{r|d} \frac{r \mu(d/r)}{1-q^r} \times n^{\alpha-1} 
     \right) \\ 
     & = [q^x]\left(\sum_{d=1}^x H_{\floor{\frac{x}{d}}}^{(1-\alpha)} \cdot d^{\alpha-1} \times 
     \sum_{r|d} \frac{r \mu(d/r)}{1-q^r}\right) \\ 
     & = \sum_{d=1}^x H_{\floor{\frac{x}{d}}}^{(1-\alpha)} \cdot d^{\alpha-1} \left( 
     \sum_{r|(d,x)} r \mu(d/r)\right). 
\end{align*} 
Here, we are explicitly employing the well-known identity expanding the cyclotomic polynomials, 
$\Phi_n(q)$, cited in \eqref{eqn_CyclotomicFactorization_stmt_v1} 
to obtain our second main characterization of the 
sum-of-divisor sums. Specifically, forming the logarithmic derivatives of the divisor product 
forms of $\Phi_n(q)$ implied by this identity allows us to relate each divisor function, 
$\sigma_{\alpha}(x)$, to the sums over harmonic numbers of generalized orders and to the 
Ramanujan sums, $c_d(x)$, defined in the previous remark given in this section. 
These steps then naturally lead us to the claimed identities stated above. 
\end{proof} 

\begin{remark}[Conjectures at related divisor sum identities] 
Another identity for the generalized sum-of-divisors functions which we have obtained based on 
computational experiments involving the last step in the previous derivation steps, and which 
we only conjecture here, is given by 
\begin{equation} 
\label{eqn_conj_identity_for_sod} 
\sigma_{\alpha+1}(x) = \sum_{d|x} d^{\alpha+1} \left(\sum_{k=1}^{\floor{\frac{x}{d}}} 
     \mu(k) k^{\alpha} H_{\floor{\frac{x/d}{k}}}^{(-\alpha)}\right). 
\end{equation} 
Since the forward differences of the harmonic numbers at ratios of floored arguments 
with respect to $k$ are easily shown to satisfy the relation 
\[
\Delta\left[H_{\floor{\frac{x+1}{k}}}^{(-\alpha)}\right](k) - 
\Delta\left[H_{\floor{\frac{x}{k}}}^{(-\alpha)}\right](k) = 
     -\left(\frac{x+1}{k}\right)^{\alpha} \Iverson{k|x+1} + 
     \left(\frac{x+1}{k+1}\right)^{\alpha} \Iverson{k+1|x+1}, 
\]
we are able to relate the $\alpha$-weighted cases of the Mertens summatory functions 
\[
M_{\alpha}(x) = \sum_{n \leq x} \frac{\mu(n)}{n^{\alpha}}, 
\]
defined in the earlier remarks connecting Ramanujan sums to our 
primary new results in this article by applying partial summation. 
\end{remark} 

\begin{table}[ht!]

\begin{center} 
\[
\boxed{
\begin{array}{l|l} \hline 
n & \PL_{\alpha}(n) \\ \hline
0 & = 1 \\ 
1 & = 1 \\ 
2 & = 1 + 2^{\alpha} \\ 
3 & = 1 + 2^{\alpha} + 3^{\alpha} \\ 
4 & = 1 + 3^{\alpha} + \frac{3 \cdot 2^{\alpha}}{2}\left(1+2^{\alpha}\right) \\ 
5 & = 1 + 3^{\alpha} + 5^{\alpha} + 6^{\alpha} + 
     \frac{3 \cdot 2^{\alpha}}{2}\left(1+2^{\alpha}\right) \\ 
6 & = 1 + 5^{\alpha} + 2^{\alpha+1}\left(2^{\alpha} + 3^{\alpha}\right) + 
     \frac{3^{\alpha}}{2}\left(3 + 3^{\alpha}\right) + 
     \frac{2^{\alpha}}{6}\left(11 + 7 \cdot 4^{\alpha}\right). \\ \hline
\end{array}
}
\]
\end{center}

\caption{\textbf{Generalized Planar Partitions.} 
         The first few values of the generalized planar 
         partition functions for symbolic $\alpha \in \mathbb{Z}^{+}$.}
\label{table_GenPlanarPartitions}

\end{table}  

\subsubsection{Example: Generalized forms of planar partitions and exponentials of the 
               sum-of-divisors generating functions} 

We have some additional new consequences of the main theorems in this article via their 
relation to logarithmic derivatives of special polynomials.
In particular, we can restate Theorem \ref{theorem_ExactFormulas_v2} 
in terms of logarithmic derivatives of products of the cyclotomic polynomials via 
\eqref{eqn_CyclotomicFactorization_stmt_v1}, \eqref{eqn_KeyCyclotomicPolyProdIdent_ForLLambdaFn}, and 
Definition \ref{def_NotationProductsLogDer} 
in the following forms for $x \geq 0$ and any (say) integers $\alpha \geq 0$:
\begin{align} 
\label{eqn_DivisorSigmaAlpha_Idents-ThmReStmt_v1} 
\frac{\sigma_{\alpha+1}(x)}{x} & = 
     [q^x]\left(\sum_{n=1}^x \log\left(\frac{1}{q^n-1}\right) n^{\alpha}\right)  = 
     [q^x]\log\left(\frac{1}{\prod\limits_{1 \leq n \leq x} (1-q^n)^{n^{\alpha}}}\right). 
\end{align} 
The expansions of the sum-of-divisors functions 
given in \eqref{eqn_DivisorSigmaAlpha_Idents-ThmReStmt_v1} above then 
follow from writing \eqref{eqn_Intro_LambertSeriesTerm_exps_v1} in the form 
\begin{equation*} 
\frac{nq}{1-q^n} = nq - q \sum_{d|n} \widetilde{\Phi}_d\left(q^{-1}\right), 
\end{equation*} 
and then applying the cited identities to the definition of the modified 
logarithmic derivative variants of the cyclotomic polynomials defined in the 
introduction. 

We enumerate the generalized planar partitions $\PL_{\alpha}(n)$ 
for non-negative integers $\alpha$ and all $n \geq 0$ by the following class of 
parameterized (in $\alpha$) generating functions 
(\cf \seqnum{A000219}, \seqnum{A000991}, \seqnum{A001452}, \seqnum{A002799}, 
\seqnum{A023871}--\seqnum{A023878}, \seqnum{A144048} and 
\seqnum{A225196}--\seqnum{A225199}): 
\begin{align*} 
\PL_{\alpha}(n) & := [q^n]\left(\prod_{n \geq 1} \frac{1}{(1-q^n)^{n^{\alpha}}}\right) \\ 
     & \phantom{:} = [q^n]\exp\left(\sum_{m \geq 1} \frac{\sigma_{\alpha+1}(m) q^m}{m}\right). 
\end{align*} 
The details of the generating-function-based argument based on an 
extension of the logarithmic derivatives implicit to the expansions in 
\eqref{eqn_Intro_LambertSeriesTerm_exps_v1} allow us to effectively generalize the first 
two cases of the convoluted recurrence relations for $p(n)$ and $\PL(n)$ in the next forms. 
\[
n \cdot \PL_{\alpha}(n) = \sum_{k=1}^n \sigma_{\alpha+1}(k) \PL_{\alpha}(n-k),\ n \geq 1. 
\]
Table \ref{table_GenPlanarPartitions} lists computations of the first few special case values of 
these sequences for symbolic $\alpha$. 
We immediately can see that when $\alpha := 0, 1$, these partition functions correspond to the 
\emph{partition function} $p(n)$ and the (ordinary) sequence of \emph{planar partitions} 
$\PL(n)$, respectively (\seqnum{A000041}, \seqnum{A000219}). 
Thus without motivating this class of recursive identities relating generalized 
planar partitions and the sum-of-divisors functions by physical or geometric interpretations to 
counting such partition numbers (for example, as are known in the case of the ordinary 
planar partitions $\PL(n)$), we have managed to relate the special multiplicative divisor functions 
we study in this article to a distinctly more additive flavor of number theoretic functions\footnote{ 
    We remark that 
    such relations between multiplicative number theory and the more additive theories of 
    partitions and special functions are decidedly rare in the supporting literature. 
    See \cite{MERCA-SCHMIDT-AMM,SCHMIDT-LAMBERTSER-MATRIXEQNS} for further examples and 
    references to other works relating these two branches of number theory 
    (i.e., the additive and multiplicative structures of the respective functions involved). 
}.

\section{Comparisons to other exact formulas for partition and divisor functions} 
\label{Section_CompsToOtherExactFormulas} 

\subsection{Exact formulas for the divisor and sums-of-divisor functions} 

\subsubsection{Finite sums and trigonometric series identities} 

There is an infinite series for the ordinary sums-of-divisors function, 
$\sigma(n)$, due to Ramanujan in the form of \citep[\S 9, p.\ 141]{HARDY-RAMANUJAN} 
\begin{align} 
\label{eqn_SigmaFn_InfTrigSumExp_Ident} 
\sigma(n) = \frac{n \pi^2}{6}\left[1 + \frac{(-1)^n}{2^2} + 
     \frac{2\cos\left(\frac{2}{3} n \pi\right)}{3^2} + 
     \frac{2\cos\left(\frac{2}{5} n \pi\right) + 
     2\cos\left(\frac{4}{5} n \pi\right)}{5^2} + \cdots\right] 
\end{align} 
In similar form, we have a corresponding infinite sum providing an exact formula 
for the divisor function expanded in terms of the functions $c_q(n)$ defined in 
\citep[\S 9]{HARDY-RAMANUJAN} of the form (\cf Remark \ref{remark_RamanujansSum}) 
\begin{align} 
\label{eqn_DivFn_InfCqnSumExp_Ident} 
d(n) & = - \sum_{k \geq 1} \frac{c_k(n)}{k} \log(k) = 
     -\frac{c_2(n)}{2} \log 2 - \frac{c_3(n)}{3} \log 3 - \frac{c_4(n)}{4} \log 4 
     - \cdots. 
\end{align} 
Recurrence relations between the generalized sum-of-divisors functions are 
proved in the references 
\citep{SCHMIDT-DIVSIGMA-IDENTS-2017,SCHMIDT-LAMBERTSER-MATRIXEQNS}. 
There are also a number of known convolution sum identities involving the 
sum-of-divisors functions which are derived from their relations to Lambert series 
and Eisenstein series. 

\subsubsection{Exact formulas for sums of the divisor function} 

Exact formulas for the divisor function, $d(n)$, of a much different 
characteristic nature are expanded in the results of 
\citep{CARELLA-EXPLFORMULA-DIVFN}. 
First, we compare our finite sum results with the infinite sums in the 
(weighted) Voronoi formulas for the partial sums over the divisor function 
expanded as 
\begin{align*} 
\frac{x^{\nu-1}}{\Gamma(\nu)} \sum_{n \leq x} \left(1-\frac{n}{x}\right)^{\nu-1} 
     d(n) & = \frac{x^{\nu-1}}{4\Gamma(\nu)} + 
     \frac{x^{\nu} \left(\log x+\gamma-\psi(1+\nu)\right)}{\Gamma(\nu+1)} - 
     2\pi x^{\nu} \sum_{n \geq 1} d(n) F_{\nu}\left(4\pi\sqrt{nx}\right) \\ 
\sum_{n \leq x} d(n) & = \frac{1}{4} + \left(\log x + 2\gamma-1\right) x \\ 
     & \phantom{= \frac{1}{4}\ } - 
     \frac{2\sqrt{x}}{\pi} \sum_{n \geq 1} \frac{d(n)}{\sqrt{n}}\left( 
     K_1\left(4\pi\sqrt{nx}\right) + \frac{\pi}{2} 
     Y_1\left(4\pi\sqrt{nx}\right)\right), 
\end{align*} 
where $F_{\nu}(z)$ is some linear combination of the \emph{Bessel functions}, 
$K_{\nu}(z)$ and $Y_{\nu}(z)$, and $\psi(z)$ is the \emph{digamma function}. 
A third identity for the partial sums, or average order, of the divisor function 
is expanded directly in terms of the Riemann zeta function, $\zeta(s)$, and 
its non-trivial zeros $\rho$ in the next equation. 
\begin{align*} 
\sum_{n \leq x} d(n) & = -\frac{\pi^2}{12} + (\log x+2\gamma-1)x + 
     \frac{\pi^2}{3} \sum_{\substack{\rho: \zeta(\rho) = 0 \\ 
     \rho \neq -2, -4, -6, \ldots}} 
     \frac{\zeta(\rho/2)^2}{\rho \zeta^{\prime}(\rho)} x^{\rho/2} \\ 
     & \phantom{=-\frac{\pi^2}{12}\ } - 
     \frac{\pi^2}{6} \sum_{n \geq 0} \frac{\zeta(-(2n+1)) x^{-(2n+1)}}{ 
     (2n+1) \zeta^{\prime}(-(2n+1))} 
\end{align*} 
While our new exact sum formulas in Theorem \ref{theorem_ExactFormulas_v2} 
are deeply tied to the prime numbers $2 \leq p \leq x$ for any $x$, we 
once again observe that 
the last three infinite sum expansions of the partial sums over the divisor function 
are of a much more distinctive character than our new exact finite sum formulas 
proved by the theorem. 

\subsection{Comparisons with other exact formulas for special functions} 

\subsubsection{Rademacher's formula for the partition function $p(n)$} 

Rademacher's famous exact formula for the partition function $p(n)$ when 
$n \geq 1$ is stated as \citep{SILLS-RADEMACHER} 
\begin{align*} 
p(n) & = \frac{1}{\pi\sqrt{2}} \sum_{k \geq 1} A_k(n) \sqrt{k} \frac{d}{dn}\left[ 
     \frac{\sinh\left(\frac{\pi}{k} \sqrt{\frac{2}{3}\left(n-\frac{1}{24}\right)} 
     \right)}{\sqrt{n-\frac{1}{24}}}\right], 
\end{align*} 
where 
\begin{align*} 
A_k(n) & := \sum_{\substack{0 \leq h < k \\ \gcd(h, k) = 1}} 
     e^{\pi\imath s(h, k)-2\pi\imath nh / k}, 
\end{align*} 
is a \emph{Kloosterman-like sum} and $s(h, k)$ and $\omega(h, k)$ are the 
(exponential) \emph{Dedekind sums} defined by 
\begin{align} 
\label{eqn_Dedekind_sum_functions_shk_omegahk} 
s(h, k) & := \sum_{r=1}^{k-1} \frac{r}{k}\left(\frac{hr}{k} - 
     \left\lfloor \frac{hr}{k} \right\rfloor - \frac{1}{2}\right) \\ 
\notag 
\omega(h, k) & := \exp\left(\pi\imath \cdot s(h, k)\right). 
\end{align} 
In comparison to other somewhat related 
formulas for special functions, we note that 
unlike Rademacher's series for the partition function, $p(n)$, expressed in terms 
of finite Kloosterman-like sums, our expansions require only a sum over 
finitely-many primes $p \leq x$ to evaluate the special function 
$\sigma_{\alpha}(x)$ at $x$. 
For comparison with the previous section, 
we also note that the partition function $p(n)$ is related to the sums-of-divisor 
function $\sigma(n)$ through the convolution identity 
\begin{align*} 
n p(n) & = \sum_{k=1}^n \sigma(n-k) p(k). 
\end{align*} 
There are multiple recurrence relations that can be given for $p(n)$ including the 
following expansions: 
\begin{align*} 
p(n) & = p(n-1) + p(n-2) - p(n-5) - p(n-7) + p(n-12) + p(n-15) - p(n-22) - \cdots \\ 
     & = 
     \sum_{\substack{k=\left\lfloor -(\sqrt{24n+1}+1) / 6 \right\rfloor \\ 
           k \neq 0}}^{\left\lfloor (\sqrt{24n+1}-1) / 6 \right\rfloor}
     (-1)^{k+1} p\left(n-\frac{k(3k+1)}{2}\right). 
\end{align*} 

\subsubsection{Rademacher-type infinite sums for other partition functions} 

\begin{example}[Overpartitions] 
An \emph{overpartition} of an integer $n$ is defined to be a representation of $n$ 
as a sum of positive integers with non-increasing summands such that the 
last instance of a given summand in the overpartition may or may not have an 
overline bar associated with it. 
The total number of overpartitions of $n$, $\bar{p}(n)$, is generated by 
\begin{align*} 
\sum_{n \geq 0} \bar{p}(n) q^n & = \prod_{m \geq 1} \frac{1+q^m}{1-q^m} = 
     1 + 2q + 4q^2+8q^3+14q^4+24q^5+40q^6 + \cdots. 
\end{align*} 
A convergent Rademacher-type infinite series providing an exact formula for the 
partition function $\bar{p}(n)$ is given by 
\begin{align*} 
\bar{p}(n) & = \sum_{\substack{k \geq 1 \\ 2 \nmid k}} 
     \sum_{\substack{0 \leq h < k \\ \gcd(h, k) = 1}} \frac{\sqrt{k}}{2\pi} 
     \frac{\omega(h, k)^2}{\omega(2h, k)} e^{-2\pi\imath nh / k} 
     \frac{d}{dn}\left[\frac{\sinh\left(\frac{\pi\sqrt{n}}{k}\right)}{\sqrt{n}} 
     \right]. 
\end{align*} 
\end{example} 

\begin{example}[Partitions where no odd part is repeated] 
Let the function $p_{\od}(n)$ denote the number of partitions of a non-negative 
integer $n$ where no odd component appears more than once. 
This partition function variant is generated by the infinite product 
\begin{equation*} 
\sum_{n \geq 0} p_{\od}(n) q^n = \prod_{m \geq 1} \frac{1+q^{2m-1}}{1-q^{2m}} = 
     1 + x+x^2+2x^3+3x^4+4x^5+5x^6+7x^7+10x^8 + \cdots. 
\end{equation*} 
We have a known Rademacher-type sum exactly generating $p_{\od}(n)$ for each 
$n \geq 0$ expanded in the form of 
\begin{align*} 
p_{\od}(n) & = \frac{2}{\pi} \sum_{k \geq 1} \sqrt{k\left(1+(-1)^{k+1} + 
     \left\lfloor \frac{\gcd(k, 4)}{4} \right\rfloor\right)} \times \\ 
     & \phantom{=\ } \times 
     \sum_{\substack{0 \leq h < k \\ \gcd(h, k) = 1}} 
     \frac{\omega(h, k) \omega\left(\frac{4h}{\gcd(k, 4)}, \frac{k}{\gcd(k, 4)} 
     \right)}{\omega\left(\frac{2h}{\gcd(k, 2)}, \frac{k}{\gcd(k, 2)}\right)} 
     e^{-2\pi\imath nh / k} \frac{d}{dn}\left[\frac{\sinh\left( 
     \frac{\pi\sqrt{\gcd(k, 4) (8n-1)}}{4k}\right)}{ 
     \sqrt{8n-1}}\right]. 
\end{align*} 
\end{example} 

Each of these three partition function variants are special cases of the 
partition function $p_r(n)$ which is generated by the infinite product 
$\prod_{m \geq 1} (1+x^m) / (1-x^{2^r m})$. Another more general, and more 
complicated Rademacher-like sum for this $r$ partition function is given in 
\citep[\S 2]{SILLS-RADEMACHER}, though we do not restate this result in 
this section. 
As we remarked in the previous section, these infinite expansions are still unlike the 
prime-related exact finite sum formulas we have proved as new results in the 
article. 

\subsubsection{Sierpi\'nski's Voronoi-type formula for the sum-of-squares function} 

The \emph{Gauss circle problem} asks for the count of the number of lattice points, 
denoted $N(R)$, inside a circle of radius $R$ centered at the origin. 
It turns out that we have exact formulas for $N(R)$ expanded as both a sum over the 
sum-of-squares function and in the form of 
\begin{equation*} 
N(R) = 1 + 4 \lfloor R \rfloor + 4 \sum_{i=1}^{\lfloor R \rfloor} 
     \left\lfloor \sqrt{R^2-i^2} \right\rfloor. 
\end{equation*} 
The Gauss circle problem is closely-related to the Dirichlet divisor problem 
concerning sums over the divisor function, $d(n)$ 
\citep[\S 1.4]{CARELLA-EXPLFORMULA-DIVFN} \cite{ONCIRCLE-AND-DIVISOR}. 
It so happens that the sum-of-squares function, $r_2(n)$, satisfies the next 
Sierpi\'nski formula which is similar and of Voronoi type to the first two identities 
for the divisor function given in the previous subsection. However, we note that the 
Voronoi-type formulas in the previous section are expanded in terms of certain 
Bessel functions, where our identity here requires the alternate special function 
\begin{equation*} 
J_{\nu}(z) = \sum_{n \geq 0} \frac{(-1)^n}{n! \cdot \Gamma(\nu+n+1)} 
     \left(\frac{z}{2}\right)^{\nu+2n}, 
\end{equation*} 
which is an unconditionally convergent series for all $0 < |z| < \infty$ and any 
selection of the parameter $\nu \in \mathbb{C}$. 
\begin{align*} 
\sum_{n \leq x} r_2(n) & = \pi x + \sqrt{x} \sum_{n \geq 1} 
     \frac{r_2(n)}{\sqrt{n}} J_1\left(2\pi \sqrt{nx}\right) 
\end{align*} 
As in the cases of the weighted Voronoi formulas involving the divisor function 
stated in the last section, we notice that this known exact formula for the 
special function $r_2(n)$ has a much different nature to its expansion than our 
new exact finite sum formulas for $\sigma_{\alpha}(n)$. 
We do, however, have an analog to the infinite series for the classical divisor 
functions in \eqref{eqn_SigmaFn_InfTrigSumExp_Ident} and 
\eqref{eqn_DivFn_InfCqnSumExp_Ident} which is expanded in terms of 
Ramanujan's sum functions 
\begin{equation*} 
c_q(n) = \sum_{\substack{\ell m = q \\ m|n}} \mu(\ell) m, 
\end{equation*} 
as follows \citep[\S 9]{HARDY-RAMANUJAN}: 
\begin{align*} 
r_2(n) & = \pi\left(c_1(n) - \frac{c_3(n)}{3} + \frac{c_5(n)}{5} - \cdots\right). 
\end{align*} 

\section{Conclusions} 

\subsection{Summary} 

In this article, we again began by considering the building blocks of the 
Lambert series generating functions for the sum-of-divisors functions in 
\eqref{eqn_LambertSeriesOGF_def}. 
The new exact formulas for these special arithmetic functions are 
obtained in this case by our observation of the expansions of the series terms, 
$q^n / (1-q^n)$, by cyclotomic polynomials and their logarithmic derivatives. 
We note that in general it is hard to evaluate the series coefficients of 
$\widetilde{\Phi}_n(q)$ without forming the divisor sum employed in the proof of 
part (i) of Proposition \ref{prop_ComponentSum_SeriesCoeffs}. 
We employed the established, or at least easy to derive, key formulas 
for the logarithmic derivatives of the cyclotomic polynomials 
along with known formulas for reducing cyclotomic polynomials of the form 
$\Phi_{p^r m}(q)$ when $p \nmid m$ to establish the Lambert series term 
expansions in \eqref{eqn_Intro_LambertSeriesTerm_exps_v1}. 
The expansions of our new exact formulas for the generalized sum-of-divisors 
functions are deeply related to the prime numbers and the distribution of the 
primes $2 \leq p \leq x$ for any $x \geq 2$ through this construction. 

One of the other interesting results cited in the introduction provides new 
relations between the $r$-order harmonic numbers, $H_n^{(r)}$, when $r > 0$ and the 
Bernoulli polynomials when $r \leq 0$ in the form of 
\eqref{eqn_SigmaAlphax_Negative-Order_Symmetry_Ident}. 
We applied these new exact functions for $\sigma_{\alpha}(n)$, $\sigma(n)$, and 
$d(n)$ to formulate asymptotic formulas for the partial sums of the divisor 
function, or its \emph{average order}, which match more famous known asymptotic 
formulas for these sums. The primary application of the new exact formulas for 
$\sigma(n)$ provided us with a new necessary and sufficient condition 
characterizing the perfectness of a positive integer $n$ 
in Section \ref{subSection_PerfectNumbers}. 
Finally, in Section \ref{Section_CompsToOtherExactFormulas} we compared the new 
results proved within this article to other known exact formulas for the 
sum-of-divisors functions, partition functions, and other special arithmetic 
sequences. 

\subsection{Extensions of these techniques to other special Lambert series expansions} 

Lambert series are special in form because they encode as their 
series coefficients key number theoretic functions which are otherwise inaccessible by 
standard combinatorial generating function techniques. 
More generally, for any well behaved arithmetic function $f$, we can form its Lambert 
series generating function as 
\[
L_f(q) := \sum_{n \geq 1} \frac{f(n) q^n}{1-q^n} = \sum_{m \geq 1} \left( 
     \sum_{d|m} f(d)\right) q^m. 
\]
Then it is easy to extend our proofs from 
Section \ref{Section_Proofs} to show that 
\begin{align*} 
\sum_{d|x} f(d) = \sum_{k=1}^x \frac{f(k)}{k} & + 
     \sum_{p \leq x} \sum_{k=1}^{\nu_p(x)+1} \sum_{r=1}^{\floor{\frac{x}{p^k}}} 
     \frac{f(p^k \cdot r)}{r} \\ 
     & + 
     \sum_{p \leq x} \sum_{k=1}^{\nu_p(x)+1} \sum_{r=1}^{\floor{\frac{x}{2p^k}}} 
     (-1)^{\floor{\frac{2p}{p^k}}} \frac{f((2p)^k \cdot r)}{r} \\ 
     & + \sum_{d=1}^x \frac{\chi_{\PP}(d)}{d \cdot n} c_d(x) \times 
     \sum_{n=1}^{\floor{\frac{x}{d}}} f(d \cdot n). 
\end{align*} 

\subsection{Supplementary computational data for readers and reviewers} 

A \emph{Mathematica} notebook containing definitions that can be used to 
computationally verify the formulas proved in this manuscript is available 
online at the following link: \\ 
\href{https://drive.google.com/open?id=13GYUxEn5RXes6xgEPtL-BvgJA76TF4Jh}{https://drive.google.com/open?id=13GYUxEn5RXes6xgEPtL-BvgJA76TF4Jh}. \\
The functional definitions provided in this notebook are also of use to 
readers for experimental mathematics based on the contents of this article. 
In particular, this data is available to those readers who wish to extend the 
results presented here or for conjecturing new properties related to the 
generalized sum-of-divisors functions based on empirical results obtained 
with the code freely available in this supplementary reference.

\newpage 

\setcounter{section}{0} 
\renewcommand{\thesection}{\Alph{section}} 
\label{page_AppendixA_Proofs} 

\section{Appendix: The complete proof of Theorem \ref{theorem_SigmaAlphaBeta_SumFunc_Asymptotics}} 

\begin{lemma}
\label{lemma_DivSumIdents} 
For any arithmetic functions $f,g,h: \mathbb{N} \rightarrow \mathbb{C}$, and natural numbers 
$x \geq 1$, we have the following pair of divisor sum simplification identities: 
\begin{align*} 
\sum_{n=1}^x f(n) \sum_{d|n} g(d) h\left(\frac{n}{d}\right) & = 
     \sum_{d \leq x} g(d) \sum_{n=1}^{\floor{\frac{x}{d}}} 
     h(n) f(dn) \\ 
\sum_{d=1}^x f(d) \left(\sum_{r|(d,x)} g(r) h\left(\frac{d}{r}\right)\right) & = 
     \sum_{r|x} g(r) \left(\sum_{1 \leq d \leq x/r} h(d) f(rd)\right). 
\end{align*} 
\end{lemma} 

Since the proofs of these identities are not difficult, and indeed are fairly standard 
exercises, we will not prove the two formulas in Lemma \ref{lemma_DivSumIdents}. 

\begin{proof}[Proof of Theorem \ref{theorem_SigmaAlphaBeta_SumFunc_Asymptotics}]
We break down the proof into four separate tasks of estimating the component 
term sums from Theorem \ref{theorem_main_SOD_result_V2_symmetric form}. 
The sum of the dominant and error terms resulting from each of these 
cases then constitutes the proof of the result at hand. In particular, we will 
define and asymptotically analyze the formulas for the following key sums: 
\begin{align*} 
\Sigma_{\alpha,\beta}^{\prime}(x) & := 
     \sum_{n \leq x} n^{\alpha-\beta} \cdot S_1^{(-\alpha)}(n) \\ 
\Sigma_{\alpha,\beta}^{\prime\prime}(x) & := 
     \sum_{n \leq x} n^{\alpha-\beta} \cdot S_2^{(-\alpha)}(n) \\ 
\Sigma_{\alpha,\beta}^{\prime\prime\prime}(x) & := 
     \sum_{n \leq x} \tau_{-\alpha}(n) n^{\alpha-\beta}. 
\end{align*} 
The combined approximation we are after is given in terms of this notation by 
\[
\Sigma^{(\alpha, \beta)}(x) = 
     \sum_{n \leq x} n^{\alpha-\beta} \cdot H_n^{(\alpha+1)} + 
     \Sigma_{\alpha,\beta}^{\prime}(x) + \Sigma_{\alpha,\beta}^{\prime\prime}(x) + 
     \Sigma_{\alpha,\beta}^{\prime\prime\prime}(x). 
\]
\emph{(I) Leading Term Estimates:} 
For $\alpha > 0$ and $x \geq 1$, we have that the $(\alpha+1)$-order harmonic numbers satisfy
\[
H_x^{(\alpha+1)} = \zeta(\alpha+1) + O\left(x^{-(\alpha+1)}\right), 
\]
where $\zeta(s)$ is the \emph{Riemann zeta function}. 
Notice that the leading terms in the formula for $\sigma_{\alpha}(n)$ from 
Theorem \ref{theorem_main_SOD_result_V2_symmetric form} then correspond to 
\begin{align} 
\notag
\sum_{n \leq x} n^{\alpha-\beta} \cdot H_n^{(\alpha+1)} & = 
     \sum_{n \leq x} n^{\alpha-\beta}\left(\zeta(\alpha+1) + 
     O\left(\frac{1}{n^{\alpha+1}}\right)\right) \\ 
\notag 
     & = \sum_{n \leq x} \left(\zeta(\alpha+1) n^{\alpha-\beta} + 
     O\left(\frac{1}{n^{1+\beta}}\right)\right) \\ 
\notag 
     & = \zeta(\alpha+1)\left[ 
     \frac{x^{\alpha+1-\beta}}{\alpha+1-\beta} + \frac{x^{\alpha-\beta}}{2} + 
     \sum_{j=2}^{\alpha-\beta} \binom{\alpha+1-\beta}{j} B_j 
     \frac{x^{\alpha+1-\beta-j}}{\alpha+1-\beta}\right] \\ 
\label{eqn_proof_asympSigmaAlphaBeta_Terms_v0} 
     & \phantom{=\zeta(\alpha+1)\ } + 
     \begin{cases} 
     O(\log x), & \text{if } \beta = 0; \\ 
     O(1), & \text{if } \beta > 0, 
     \end{cases} 
\end{align} 
where $B_n$ are the \emph{Bernoulli numbers}. 
This establishes the leading dominant term in the asymptotic expansion which 
matches with the known classical result cited above. \\ 
\emph{(II) Second Terms Estimate:} 
Next, we can asymptotically expand the first
component sum in Theorem \ref{theorem_main_SOD_result_V2_symmetric form} as 
\begin{align*} 
S_1^{(-\alpha)}(n) & = 
     \sum_{p \leq n} \left[ 
     \sum_{k=1}^{\nu_p(n)} \frac{(p-1)}{p}\left(\frac{\zeta(\alpha+1)}{p^{\alpha k}} + 
     O\left(\frac{p^k}{n^{\alpha+1}}\right)\right) - 
     \frac{\zeta(\alpha+1)}{p^{\alpha \nu_p(n) + \alpha+1}} + 
     O\left(\frac{p^{\nu_p(n)}}{n^{\alpha+1}}\right)
     \right] \\ 
     & = 
     \sum_{p \leq n} \left[ 
     \frac{\zeta(\alpha+1)(p-1)}{p^{\alpha \nu_p(n) + 1} 
     \left(p^{\alpha}-1\right)} - 
     \frac{\zeta(\alpha+1) (p-1)}{p^{\alpha+1} \left(p^{\alpha}-1\right)}  
     \right] \Iverson{p | n} \\ 
     & \phantom{=\sum\ } - 
     \sum_{p \leq n} 
     \frac{\zeta(\alpha+1)}{p^{\alpha \nu_p(n)+\alpha+1}} + 
     O\left(\frac{1}{n^{\alpha-1} \cdot \log n}\right), 
\end{align*} 
where the last error term results by observing that $\nu_p(n) \leq \log_p(n)$. 
We can use Abel summation together with a divergent asymptotic expansion 
for the exponential integral function to determine that for real $r > 0$, 
the prime sums 
\begin{align*} 
\sum_{p \leq x} \frac{1}{p^{r+1}} & = C_r + \frac{r+1}{r \cdot x^r \log x} + 
     O\left(\frac{1}{x^{r} \log^2 x}\right) \\ 
\sum_{p \leq x} \frac{1}{p^{r+1} \log p} & = D_r + 
     \frac{(r+2)}{2 x^{r} \log x} - 
     \frac{(r+2) \log x + 1}{2 x^{r} \log^2 x} + 
     O\left(\frac{1}{x^{r} \log^2 x}\right), 
\end{align*} 
where the terms $C_r,D_r$ are absolute constants depending only on $r$. 
Now we will perform the sum over $n$ and multiply by a factor of $n^{\alpha}$ as in 
Theorem \ref{theorem_main_SOD_result_V2_symmetric form}, and then 
interchange the indices of summation to obtain that 
\begin{align*} 
\Sigma_{\alpha,\beta}^{\prime}(x) & := 
     \sum_{n \leq x} n^{\alpha-\beta} \cdot S_1^{(-\alpha)}(n) \\ 
     & \phantom{:} = 
     \sum_{p \leq x} \left[ 
     \sum_{n=p}^x n^{\alpha-\beta} \left(
     \frac{\zeta(\alpha+1)(p-1)}{p^{\alpha \nu_p(n) + 1} 
     \left(p^{\alpha}-1\right)} - 
     \frac{\zeta(\alpha+1) (p-1)}{p^{\alpha+1} \left(p^{\alpha}-1\right)} 
     \right) \Iverson{p | n} 
     \right] \\ 
     & \phantom{:=\sum\sum\ } - 
     \sum_{p \leq x} \sum_{n=p}^{x} 
     \frac{n^{\alpha-\beta} \zeta(\alpha+1)}{p^{\alpha \nu_p(n)+\alpha+1}}
     + O\left(\frac{1}{n^{\beta}}\right). 
\end{align*} 
For $\beta > 1$, the error term in the last equation corresponds to 
\[
\sum_{p \leq x} \sum_{n=p}^{x} O\left(\frac{1}{n^{\beta}}\right) = 
     \sum_{p \leq x} \left[\zeta(\beta) + O\left(\frac{1}{x^{\beta}}\right)\right] = 
     O\left(\frac{x}{\log x} + \frac{1}{x^{\beta-1} \log x}\right) = 
     O\left(\frac{x}{\log x}\right), 
\]
by applying the known asymptotic estimate for the \emph{prime counting function}, 
$\pi(x) = x / \log x + O(x / \log^2 x)$. 
At this point we must evaluate sums of the next form for fixed primes $p$: 
\begin{align*} 
T_{1,p}(x) & := \sum_{n=p}^x \frac{n^{\alpha-\beta} \Iverson{p | n}}{ 
     p^{\alpha \nu_p(n)}}.
\end{align*} 
We can use Abel summation and then form the approximations to the next 
summatory functions in the following way to approximate $T_{1,p}(x)$ for 
large $x$:
\begin{align*} 
A_{1,p}(t) & = 
     \sum_{i \leq t} \frac{\Iverson{p|i}}{p^{\alpha \nu_p(i)}} =
     \sum_{i \leq t/p} \frac{1}{p^{\alpha \nu_p(p \cdot i)}} \\ 
     & = 
     \sum_{k=1}^{\log_p(t)} \frac{1}{p^{\alpha k}} \#\{i \leq t/p: 
     \nu_p(i) = k\} \\ 
     & = 
     \sum_{k=0}^{\infty} 
     \frac{1}{p^{\alpha (k+1)}} \left(\#\{i \leq t/p: p^{k+1}|i\} - 
     \sum_{s=k+2}^{\infty} \#\{i \leq t/p: p^{s}|i\}\right) \\ 
     & = 
     \sum_{k=0}^{\infty} \frac{1}{p^{\alpha (k+1)}}\left[\frac{t}{p^{k+2}} - 
     \sum_{i=k+2}^{\infty} \frac{t}{p^{i+1}} 
     \right] \\ 
     & = 
     \frac{(p-2)t}{p(p-1)\left(p^{\alpha+1}-1\right)}. 
\end{align*} 
Then we have by Abel's summation formula that 
\begin{align*} 
T_{1,p}(x) & = x^{\alpha-\beta} \cdot A_{1,p}(x) - 
     \int_{0}^{x} A_{1,p}(t) D_t[t^{\alpha-\beta}] dt \\ 
     & = \frac{(p-2) x^{\alpha+1-\beta}}{(\alpha+1-\beta) p(p-1) 
     \left(p^{\alpha+1}-1\right)}. 
\end{align*} 
Similarly, we can estimate the asymptotic order of the sums 
\begin{align*} 
T_{2,p}(x) & := \sum_{n=p}^x \frac{n^{\alpha-\beta}}{ 
     p^{\alpha \nu_p(n) + \alpha + 1}} \\ 
     & = \frac{(p-2) x^{\alpha+1-\beta}}{(\alpha+1-\beta) (p-1) p^{\alpha+1} 
     \left(p^{\alpha+1}-1\right)}.
\end{align*} 
Thus it follows that 
\begin{align*} 
\sum_{p \leq x} \left[\frac{\zeta(\alpha+1) (p-1)}{p(p^{\alpha}-1)} T_{1,p}(x) - 
     \zeta(\alpha+1) T_{2,p}(x)\right] & \sim 
     \frac{C_1(\alpha) \zeta(\alpha+1) x^{\alpha+1-\beta}}{(\alpha+1-\beta)}. 
\end{align*} 
We can also estimate the summands 
\begin{align*} 
\sum_{p \leq x} \sum_{n=p}^x n^{\alpha-\beta} \cdot 
     \frac{\zeta(\alpha+1) (p-1)}{p^{\alpha+1} \left(p^{\alpha}-1\right)} 
     \Iverson{p | n} & = 
     \sum_{p \leq x} \sum_{n=1}^{x/p} (pn)^{\alpha-\beta} \cdot 
     \frac{\zeta(\alpha+1) (p-1)}{p^{\alpha+1} \left(p^{\alpha}-1\right)} \\ 
     & = \sum_{j=0}^{\alpha-\beta} \binom{\alpha+1-\beta}{j} 
     \frac{B_j x^{\alpha+1-\beta-j}}{(\alpha+1-\beta)} \cdot 
     \frac{\zeta(\alpha+1) (p-1)}{p^{\alpha+1-j} \left(p^{\alpha}-1\right)} \\ 
     & \sim \sum_{j=0}^{\alpha-\beta} \binom{\alpha+1-\beta}{j} 
     \frac{C_{2,j}(\alpha) B_j x^{\alpha+1-\beta-j}}{(\alpha+1-\beta)}. 
\end{align*} 
In summary, we obtain that 
\begin{align} 
\label{eqn_proof_asympSigmaAlphaBeta_Terms_v1} 
\Sigma_{\alpha,\beta}^{\prime}(x) & = 
     \frac{C_1(\alpha) \zeta(\alpha+1) x^{\alpha+1-\beta}}{(\alpha+1-\beta)} - 
     \sum_{j=0}^{\alpha-\beta} \binom{\alpha+1-\beta}{j} 
     \frac{C_{2,j}(\alpha) B_j x^{\alpha+1-\beta-j}}{(\alpha+1-\beta)} + 
     O\left(\frac{x}{\log x}\right). 
\end{align} 
\emph{(III) Third Component Terms Estimate:} 
Following from the estimates in the previous equations, we see that 
the second component sum from Theorem \ref{theorem_main_SOD_result_V2_symmetric form} 
is similar to the constructions of the sums we estimated last in (II), weighted 
by an additional term of $(-1)^n \times 2^{-(\alpha+1)}$ when performing the last 
(outer) sum on $n \leq x$. In particular, we have that 
\begin{align*} 
\Sigma_{\alpha,\beta}^{\prime\prime}(x) & = 
     \sum_{n \leq x} n^{\alpha-\beta} S_2^{(-\alpha)}(n) \\ 
     & = \sum_{p \leq x} \left[ 
     \sum_{n=p}^x \frac{(-1)^n \cdot n^{\alpha-\beta}}{2^{\alpha+1}} \left(
     \frac{\zeta(\alpha+1)(p-1)}{p^{\alpha \nu_p(n) + 1} 
     \left(p^{\alpha}-1\right)} - 
     \frac{\zeta(\alpha+1) (p-1)}{p^{\alpha+1} \left(p^{\alpha}-1\right)} 
     \right) \Iverson{p | n} 
     \right] \\ 
     & \phantom{:=\sum\sum\ } - 
     \sum_{p \leq x} \sum_{n=p}^{x} 
     \frac{(-1)^n \cdot n^{\alpha-\beta} \zeta(\alpha+1)}{2^{\alpha+1} 
     p^{\alpha \nu_p(n)+\alpha+1}} + 
     O\left(\frac{1}{n^{\beta}}\right). 
\end{align*} 
Hence, our estimates in this case boil down to summing the following formulas 
by Abel summation for large $x$: 
\[
U_{1,p}(x) := \sum_{n=p}^{x} \frac{(-1)^n \cdot n^{\alpha-\beta} \Iverson{p | n}}{ 
     p^{\alpha \nu_p(n)}}. 
\]
To bound the main and error terms in this sum for large enough $x$, we proceed 
as before to form the summatory functions 
\begin{align*} 
B_{1,2}(t) & = \sum_{i \leq t} \frac{(-1)^i \cdot \Iverson{2|i}}{ 
     2^{\alpha \nu_2(i)}} = \sum_{1 \leq i \leq t/2} 
     \frac{1}{2^{\alpha \nu_2(2i)}} \\ 
     & = \sum_{k=1}^{\log_2(t)} \frac{1}{2^{\alpha k}} \#\{i \leq t/2: 
     \nu_2(i) = k\} \\ 
     & = 
     \sum_{k=0}^{\infty} 
     \frac{1}{2^{\alpha (k+1)}} \left(\#\{i \leq t/2: 2^{k+1}|i\} - 
     \sum_{s=k+2}^{\infty} \#\{i \leq t/2: 2^{s}|i\}\right) \\ 
     & = \frac{(p-2) t}{(p-1) \left(p^{\alpha+1}-1\right)} 
     \Biggr\rvert_{p=2} \\ 
     & = 0, 
\end{align*} 
and for $p \geq 3$ an odd prime, we expand analogously as 
\begin{align*} 
B_{1,p}(t) & = \sum_{i \leq t} \frac{(-1)^i \cdot \Iverson{p|i}}{ 
     p^{\alpha \nu_p(i)}} = - 
     \sum_{i \leq t/p} \frac{1}{p^{\alpha \nu_p(p \cdot i)}} \\ 
     & = 
     -\sum_{k=1}^{\log_p(t)} \frac{1}{p^{\alpha k}} \#\{i \leq t/p: 
     \nu_p(i) = k\} \\ 
     & = 
     -\sum_{k=0}^{\infty} 
     \frac{1}{p^{\alpha (k+1)}} \left(\#\{i \leq t/p: p^{k+1}|i\} - 
     \sum_{s=k+2}^{\infty} \#\{i \leq t/p: p^{s}|i\}\right) \\ 
     & = 
     -\sum_{k=0}^{\infty} \frac{1}{p^{\alpha (k+1)}}\left[\frac{t}{p^{k+2}} - 
     \sum_{i=k+2}^{\infty} \frac{t}{p^{i+1}} 
     \right] \\ 
     & = -\frac{(p-2) t}{p(p-1) \left(p^{\alpha+1}-1\right)}. 
\end{align*} 
It follows from Abel's summation formula that 
\begin{align*} 
U_{1,p}(x) & = -\frac{(p-2) x^{\alpha+1-\beta}}{(\alpha+1-\beta) p(p-1) 
     \left(p^{\alpha+1}-1\right)} \Iverson{p \geq 3}. 
\end{align*} 
Then we can pick out the similar terms from these sums as we did in the 
previous derivations from step (II) and estimate the 
resulting formula in the form of 
\begin{align*} 
U_{2,p}(x) & = \sum_{n=p}^x \frac{(-1)^n \cdot n^{\alpha-\beta}}{2^{\alpha+1}  
     p^{\alpha \nu_p(n) + \alpha + 1}} \\ 
     & = -\frac{(p-2) x^{\alpha+1-\beta}}{(\alpha+1-\beta) (p-1) p^{\alpha+1} 
     \left(p^{\alpha+1}-1\right)} \Iverson{p \geq 3}.
\end{align*} 
When we sum the corresponding terms in these two functions over all odd primes 
$p \leq x$, we obtain that 
\begin{align*} 
\sum_{3 \leq p \leq x} \left[\frac{\zeta(\alpha+1) (p-1)}{p(p^{\alpha}-1)} U_{1,p}(x) - 
     \zeta(\alpha+1) U_{2,p}(x)\right] & \sim 
     -\frac{C_3(\alpha) \zeta(\alpha+1) x^{\alpha+1-\beta}}{(\alpha+1-\beta)}. 
\end{align*} 
To complete the estimate for the formula in this section (III), it remains to 
perform the summation estimates \citep[\cf \S 24.4(iii)]{NISTHB} 
\begin{align*} 
U_{3,p}(x) & := 
     \sum_{n=p}^x \frac{(-1)^n \cdot n^{\alpha-\beta}}{2^{\alpha+1}} \left(
     \frac{\zeta(\alpha+1) (p-1)}{p^{\alpha+1} \left(p^{\alpha}-1\right)} 
     \right) \Iverson{p | n} \\ 
     & \phantom{:} = 
     \sum_{n=1}^{x/p} \frac{(-1)^n \cdot (pn)^{\alpha-\beta}}{2^{\alpha+1}} \left(
     \frac{\zeta(\alpha+1) (p-1)}{p^{\alpha+1} \left(p^{\alpha}-1\right)} 
     \right) \\ 
     & \phantom{:} = \left[\sum_{k=0}^{\alpha-\beta} \binom{\alpha-\beta}{k} 
     \frac{(-1)^{\alpha-\beta} E_k}{2^{k+\alpha+2}} \left(\frac{x}{p} - 
     \frac{1}{2}\right)^{\alpha-\beta-k} + \frac{(-1)^{\alpha-\beta}}{2^{2\alpha+2-\beta}} 
     \sum_{k=0}^{\alpha-\beta} \binom{\alpha-\beta}{k} (-1)^k E_k\right] \times \\ 
     & \phantom{:= \Biggl[\sum\ } \times \left( 
     \frac{\zeta(\alpha+1) (p-1)}{p^{\beta+1} \left(p^{\alpha}-1\right)} 
     \right) 
\end{align*} 
Thus in total, for estimate (III) we have that 
\begin{align} 
\notag 
\Sigma_{\alpha,\beta}^{\prime\prime}(x) & \sim 
     -\frac{C_3(\alpha) \zeta(\alpha+1) x^{\alpha+1-\beta}}{(\alpha+1-\beta)} + 
     \sum_{r=0}^{\alpha-\beta} C_{4,r}(\alpha, \beta) x^r \\ 
\label{eqn_proof_asympSigmaAlphaBeta_Terms_v2} 
     & \phantom{\sim\ } + 
     \frac{(-1)^{\alpha-\beta}}{2^{2\alpha+2-\beta}} 
     \sum_{k=0}^{\alpha-\beta} \binom{\alpha-\beta}{k} (-1)^k E_k \times 
     \sum_{p \geq 2} 
     \frac{\zeta(\alpha+1) (p-1)}{p^{\beta+1} \left(p^{\alpha}-1\right)} 
     + o(1) 
\end{align} 
\emph{(IV) Last Remaining Estimates (Tau Divisor Sum Terms):} 
Finally, we have only one component in the sums from 
Theorem \ref{theorem_main_SOD_result_V2_symmetric form} left to estimate. 
Namely, we must bound the sums 
\[
\Sigma_{\alpha,\beta}^{\prime\prime\prime}(x) := 
     \sum_{n \leq x} \tau_{-\alpha}(n) n^{\alpha-\beta} = 
     \sum_{n \leq x} \sum_{d=1}^n \frac{n^{\alpha-\beta}}{d^{\alpha+1}} 
     H_{\floor{\frac{n}{d}}}^{(\alpha+1)} c_d(n) \chi_{\PP}(d). 
\]
Now by expanding the previous sums according to the divisor sum identities 
in Lemma \ref{lemma_DivSumIdents}, we obtain that 
\begin{align} 
\notag 
\Sigma_{\alpha,\beta}^{\prime\prime\prime}(x) & = 
     \sum_{n \leq x} \sum_{d=1}^n \left(\sum_{r|(d,n)} r \mu(d/r)\right) 
     H_{\floor{\frac{n}{d}}}^{(\alpha+1)} \frac{\chi_{\PP}(d)}{d^{\alpha+1}} 
     n^{\alpha-\beta} \\ 
\notag 
     & = 
     \sum_{n \leq x} \sum_{d=1}^n \left(\sum_{r|(d,n)} r \mu(d/r)\right) 
     H_{\floor{\frac{n}{d}}}^{(\alpha+1)} \frac{n^{\alpha-\beta}}{d^{\alpha+1}} \\ 
\notag 
     & \phantom{=\ } - 
     \sum_{n \leq x} \sum_{p \leq n} \sum_{k=1}^{\log_p(n)} C_{p^k}(n) \left( 
     \zeta(\alpha+1) + O\left(\frac{p^{(\alpha+1)k}}{n^{\alpha+1}}\right)\right) 
     \frac{n^{\alpha-\beta}}{p^{(\alpha+1)k}} \\ 
\notag & = 
     \sum_{n \leq x} \sum_{r|n} \sum_{1 \leq d \leq r} 
     \frac{r^{\alpha}}{n^{\beta}} \frac{\mu(d)}{d^{\alpha+1}} 
     H_{\floor{\frac{r}{d}}}^{(\alpha+1)} \\ 
\label{eqn_proof_SigmaPPP_last_eq} 
     & \phantom{=\ } - 
     \sum_{p \leq x} \sum_{n=1}^{x / p} 
     \sum_{k=1}^{\nu_p(n)-1} p^k \times 
     \left(\zeta(\alpha+1) 
     \frac{(p^kn)^{\alpha-\beta}}{p^{(\alpha+1)k}} + O\left(\frac{1}{n^{\beta+1}}
     \right)\right) 
\end{align} 
In the transition from the second to last to the previous equation, we have used a 
known fact about the \emph{Ramanujan sums}, $c_{p^k}(n)$, at prime powers. 
Namely, $c_{p^k}(n) = 0$ whenever $p^{k-1} \nmid n$, and where the function is 
given by $c_{p^k}(n) = -p^{k-1}$ if $p^{k-1} | n$, but $p^k \nmid n$ or 
$c_{p^k}(n) = \phi(p^k) = p^k-p^{k-1}$ if $p^k | n$. 
For the first sum in \eqref{eqn_proof_SigmaPPP_last_eq}, we can extend the divisor sum 
over $r$ to all $1 \leq r \leq n$, take the absolute values of all M\"obius function 
terms over $d$, and observe that the resulting formula we obtain by applying 
summation by parts corresponds to the identity that 
$$H_{\floor{\frac{r+1}{d}}}^{(m)} - H_{\floor{\frac{r}{d}}}^{(m)} = 
       \left(\frac{r}{d}\right)^{-m} \Iverson{d|r}.$$ 
Then we see that this first sum is bounded by 
\begin{align*} 
\sum_{n \leq x} \sum_{r|n} \sum_{1 \leq d \leq r} 
     \frac{r^{\alpha}}{n^{\beta}} \frac{\mu(d)}{d^{\alpha+1}} 
     H_{\floor{\frac{r}{d}}}^{(\alpha+1)} & = O\left( 
     \sum_{n \leq x} \sum_{r=1}^n \left(r^{\alpha+1} + O(r^{\alpha})\right) \sum_{d|r} 
     \frac{|\mu(d)|}{d^{\alpha+1}} \left(\frac{d}{r}\right)^{\alpha+1}  
     \right) \\ 
     & = O\left(\sum_{n \leq x} \frac{1}{n^{\beta}} \sum_{r=1}^n \sum_{d|r} 
     |\mu(d)|\right) \\ 
     & = O\left(\sum_{n \leq x} \frac{1}{n^{\beta}} \sum_{r=1}^n 2^{\omega(r)}\right) \\ 
     & = O\left(\sum_{n \leq x} \frac{1}{n^{\beta}} \sum_{k \geq 1} 2^{k} \cdot 
     \#\{1 \leq r \leq n: \omega(r) = k\}\right). 
\end{align*} 
Now we can draw upon the result of Erd\"os in \cite{ERDOS-PRIMEKX} to approximately 
sum the right-hand-side of the last equation as follows for integers $\beta > 1$: 
\begin{align*} 
O\left(\sum_{n \leq x} 2(1+o(1)) n^{1-\beta} \log n\right) & = 
     \begin{cases} 
     O\left(\log^2 x\right), & \text{ if $\beta = 2$; } \\ 
     O\left(\frac{2(\beta-2)\log x+1}{(\beta-2)^2 x^{\beta-2}}\right), & 
     \text{ if $\beta \geq 3$. } 
     \end{cases} 
\end{align*} 
Finally, to complete the estimate of (IV), we still need to bound the 
order of the second sum in \eqref{eqn_proof_SigmaPPP_last_eq} as 
\begin{align*} 
V_{2,p}(x) & := \sum_{n=1}^{x/p} 
     \sum_{k=1}^{\nu_p(n)-1} p^k \times 
     \left(\zeta(\alpha+1) 
     \frac{(p^kn)^{\alpha-\beta} \Iverson{p|n}}{p^{(\alpha+1)k}} + 
     O\left(\frac{1}{n^{\beta+1}}
     \right)\right) \\ 
     & \phantom{:} = 
     \sum_{n=1}^{x/p} \frac{\zeta(\alpha+1) n^{\alpha-\beta}}{p^{\beta \nu_p(n)}} - 
     \frac{1}{p^{\beta}} \sum_{j=0}^{\alpha-\beta} \binom{\alpha+1-\beta}{j} 
     \frac{\zeta(\alpha+1) B_j}{(\alpha+1-\beta)} \left(\frac{x}{p}\right)^{ 
     \alpha+1-\beta-j} \\ 
     & \phantom{:=\ } + 
     O\left(\frac{1}{(p-1) n^{\beta}} - \frac{p}{(p-1) n^{\beta+1}}\right), 
\end{align*} 
where we have again used the upper bound $\nu_p(n) \leq \log_p(n)$ to bound the 
error term in the previous equation. To evaluate the first component sum in the 
previous equation, we can use Abel summation much like we have in the prior estimates 
in this proof. Namely, we form the summatory functions 
\begin{align*} 
A_{3,p}(t) & = \sum_{i \leq t} \frac{\Iverson{p|i}}{p^{\beta \nu_p(i)}} = 
     \sum_{i \leq t/p} \frac{1}{p^{\beta \nu_p(pi)}} \\ 
     & = \sum_{k \geq 1} \frac{1}{p^{\beta k}} \#\{i \leq t/p: \nu_p(i) = k\} \\ 
     & = \sum_{k \geq 1} \frac{1}{p^{\beta k}} \left[ 
     \frac{t}{p^{k+1}} - \sum_{i \geq k+1} \frac{t}{p^{i+1}}\right] \\ 
     & = \frac{(p-2) t}{p(p-1) (p^{\beta+1}-1)}. 
\end{align*} 
Then the complete first sum above is given by 
\begin{align*} 
\sum_{p \leq x} \sum_{n=1}^{x/p} 
     \frac{\zeta(\alpha+1) n^{\alpha-\beta}}{p^{\beta \nu_p(n)}} & = 
     \sum_{p \leq x} \frac{\zeta(\alpha+1) 
     (p-2) x^{\alpha+1-\beta}}{(\alpha+1-\beta) p(p-1) (p^{\beta+1}-1)}. 
\end{align*} 
When we assemble the second two sums over all primes $p \leq x$, and add it to the first, 
we obtain 
\begin{align*} 
\sum_{p \leq x} V_{2,p}(x) & \sim 
     \frac{C_6(\beta) \zeta(\alpha+1) x^{\alpha+1-\beta}}{(\alpha+1-\beta)} + 
     \sum_{j=0}^{\alpha-\beta} \binom{\alpha+1-\beta}{j} 
     \frac{C_{7,j}(\alpha) \zeta(\alpha+1) B_j x^{\alpha+1-\beta-j}}{(\alpha+1-\beta)} \\ 
     & \phantom{=\ } + 
     O\left(\frac{x}{\log x}\right). 
\end{align*} 
We can make the final step of combining the asymptotic estimates derived in 
subsections (I) -- (IV) of this proof above to conclude that the result in the theorem
is true.  
\end{proof}

\end{document}